\documentclass{amsart}
\usepackage{amsmath}
\usepackage{amssymb,tikz,xcolor}
\usepackage{showkeys}

\makeatletter
\@namedef{subjclassname@2020}{\textup{2020} Mathematics Subject Classification}
\makeatother \newtheorem{thm}{Theorem}[section] \newtheorem{lem}[thm]{Lemma} \newtheorem{prop}[thm]{Proposition}
 
 \theoremstyle{remark}
\newtheorem{rem}[thm]{Remark} 

 \title{Multidimensional examples of the Metropolis algorithm}
\author{Laurent Saloff-Coste}\address{Department of Mathematics, Cornell University\\ Ithaca, NY, USA}
\email{lps2@cornell.edu}

\thanks{The first author was partially supported by NSF grant DMS--2054593}  
\author{Sophie Uluatam}

\begin{document}

\maketitle
\begin{abstract} Consider the problem of approximating a given probability distribution on the cube $[0,1]^n$ via the use of a square lattice discretization with mesh-size $1/N$ and the Metropolis algorithm. Here the dimension $n$ is fixed and we focus for the most part on the case $n=2$. In order to understand the speed of convergence of such a procedure, one needs to control the spectral gap, $\lambda$,  of the associated finite Markov chain, and how it depends on the parameter $N$.  In this work, we study basic examples for which good upper-bounds and lower-bounds on $\lambda$ can be obtained via appropriate application of path techniques. \end{abstract}

\section{Introduction} In the article \cite{What}  titled {\em What do we know about the Metropolis algorithm?}, Persi Diaconis and the first author reviewed some of the basic ideas behind the celebrated Metropolis algorithm and its quantitative analysis in the case of simple one-dimensional examples. In this sequel, we  complement this reference by treating some additional examples including basic multidimensional examples. Throughout, we require the dimension to be fixed and constants will depend (often in unspecified ways) of this fixed dimension.  In fact, we will mostly work in dimension $2$  even so the same techniques apply in higher dimensions with  only a few additional complications. For an interesting discussion of what happens when the dimension is allowed to tend to infinity, comparisons with other techniques, and pointers to the recent literature, see \cite{Kamatani}. We focus on obtaining good estimates for the spectral gap of these Markov chains using path arguments. The use of such path arguments in the context of finite reversible Markov chains is one of the topics emphasized in \cite{DiaStr}, and they play an important role in \cite{DSCcompR} and many other works. Our presentation emphasizes the idea that the finite Metropolis Markov Chains we study aim at  approximating a continuous probability distribution on a cube in $2$ or more dimensions. 

We refer the reader to \cite{What} for a discussion and earlier references regarding the Metropolis algorithm and its many variations and applications. For a different and less elementary discussion of key aspects of the use of the Metropolis algorithm, see \cite{DLME} and the references therein.

\subsection{Set-up}  \label{setup}To fix ideas, assume we want to approximate a probability measure $\mu(dx)=zf(x)dx$ for some constant $z>0$  in the  compact region $K=[0,1]^n$ in Euclidean space, that is, approximate $\mu(A)$ for reasonable subsets $A\subseteq K$, or compute $\mu(g)=\int_K g(x) \mu(dx)$ for reasonable functions $g$. In fact, one can think of the ``test problem'' of computing the normalizing constant 
$$z=\left(\int_K f(x)dx\right)^{-1}$$
appearing in the definition of the probability measure $\mu$. 

For simplicity, let us assume $K=[0,1]\times [0,1]$ and  $f>0$ is smooth (this assumption will be made more precise later) so that it makes sense to pick a large $N$ and work on the small squares of the finite grid $\{0,1/N,2/N,\dots, N/N\}^2$.  There are $N^2$ such small squares and we will parametrize them by the the integer grid points $B_N=\{1,\dots, N\}^2$ in such a way that the small square $(k_1,k_2)$ has upper-right corner at $(k_1/N,k_2/N)$. Equivalently, the center of that square is $((k_1-1/2)/N,(k_2-1/2)/N)$.
 
\begin{figure}[h]
\begin{tikzpicture}[scale=.2] 
\draw [help lines, thick] (0,0) grid (20,20);
\foreach \x in {.5,1.5,...,19.5}, \draw[red] (.5,\x) -- (19.5,\x);
 \foreach 
\x in {.5,1.5,...,19.5}, \draw[red] (\x,.5) -- (\x,19.5);

  \foreach \x in {.5,1.5,...,19.5} \foreach \y in {.5,1.5,...,19.5} 
 \draw[fill,red] (\x,\y) circle [radius=.1];

\end{tikzpicture}
\caption{Grid approximation at level $N$: The red grid has $N$ vertical lines and $N$ horizontal lines; the grey grid decomposes $[0,1]^2$ into $N^2$ little squares.}
\end{figure}
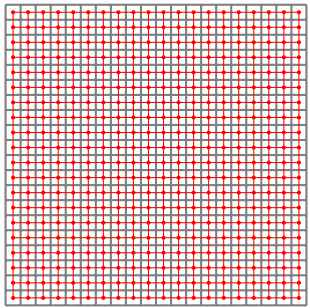

By blowing up this picture (together with the density $f$), we arrive to the finite Metropolis problem: Find a Markov chain which converges toward  the discrete probability measure $Z_NF_N$ on $B_N=\{1,\dots,N\}^2$ where $F_N$ is defined 
for $\mathbf k=(k_1,k_2)\in B_N$ by 
$$F_N(\mathbf k)= f\left(\left(\frac{k_1-\frac{1}{2}}{N},\frac{k_2-\frac{1}{2}}{N}\right)\right).$$ 
Let $Q_{\mathbf k}$ be the square of side length $1/N$ centered at $\left(\frac{k_1-\frac{1}{2}}{N},\frac{k_2-\frac{1}{2}}{N}\right)$. If $f$ is regular enough, we should be able to find $\epsilon_N>0$ such that
$$ \left|F_N(\mathbf k)- N^2\int_{Q_{\mathbf k}}fd\mu\right|<\epsilon_NF_N(\mathbf k)$$
and $$\left| \frac{1}{N^2}\sum_{\mathbf k\in B_N} F_N(\mathbf k) -\int_{[0,1]^2}fd\mu\right|<\frac{\epsilon_N }{N^2}\sum_{\mathbf k\in B_N} F_N(\mathbf k).$$
Of course, we have 
$$Z_N=\left(\sum_{\mathbf k\in B_N}F_N(\mathbf k)\right)^{-1},\;\;z=\left(\int_{[0,1]^2}fdx\right)^{-1}$$
and thus
$$\left|\frac{N^2Z_N}{z}-1\right|<\epsilon_N.$$

To find a Markov chain which converges toward  the discrete probability measure $Z_NF_N$ on $B_N=\{1,\dots,N\}^2$, we use the Metropolis algorithm with proposal based on simple random walk on the grid in continuous time.
Namely,  our proposal is 
$$Q_N(\mathbf k,\mathbf l)= \left\{\begin{array}{cl}1/4 &\mbox{  for all $\mathbf k,\mathbf l\in B_N$  with }
|k_1-l_1|+|k_2-l_2|=1,\\ 
1/4&\mbox{  if $\mathbf k=\mathbf l\in B_N$ and exactly one of $k_1,k_2$ is in $\{1,N\}$},\\
1/2 &\mbox{  if $\mathbf k=\mathbf l\in B_N$ and both $k_1,k_2$ are in $\{1,N\}$}.\end{array}
\right.$$
This proposal has the uniform distribution as its equilibrium measure. We set $|\mathbf k|=|k_1|+|k_2|.$

Throughout we use the notation $ Z\asymp T$ to signify that there are constants $0<c\le C<+\infty$ such that
$cT\le Z\le CT$. The constant $c,C$ do not depend of varying parameter (in particular do not depend on the size parameter $N$). They may depends on some additional fixed parameters such as the dimension, and on the parameters defining the functions $f,F$. 

\subsection{The Metropolis chain} Dropping the reference to $N$, assume we are given a positive function $F$ defined on $B=\{1,\dots,N\}^2$.  For our target, the discrete probability measure $\pi=ZF$, $Z^{-1}=\sum_BF$, the Metropolis Chain with kernel $M$ on $B^2$  is given as follows
$$M(\mathbf k,\mathbf l)= \left\{\begin{array}{l}\frac{1}{4}\min\{1,F(y)/F(x)\} \mbox{  for all $\mathbf k,\mathbf l\in B$  with }
|\mathbf k-\mathbf l|=1,\\ 
Q(\mathbf k,\mathbf k) +\sum_{{\mathbf m:|\mathbf k-\mathbf m|=1}\atop { F(\mathbf m)<F(\mathbf k)}}\frac{1}{4}\left(1-\frac{F(\mathbf m)}{F(\mathbf k)}\right)\mbox{  if $\mathbf k=\mathbf l\in B$}.\end{array}
\right.$$
This kernel is reversible with reversible measure $\pi=Z^{-1}F$ on $B$ in the sense that
$$\forall\,\mathbf k,\mathbf l\in B,\;\;\;M(\mathbf k,\mathbf l)\pi(\mathbf k)=M(\mathbf l,\mathbf k)\pi(\mathbf l).$$
Moreover, for $\mathbf k,\mathbf l\in B$ with $|\mathbf k-\mathbf l|=1$,
$$M(\mathbf k,\mathbf l)\pi(\mathbf k)=M(\mathbf l,\mathbf k)\pi(\mathbf l)=\frac{1}{4}\min\{\pi(\mathbf k),\pi(\mathbf l)\}.$$

It follows that the Markov operator
$$u\mapsto Mu,\;\;Mu(\mathbf k)=\sum_{\mathbf l\in B} M(\mathbf k,\mathbf l)u(\mathbf l)$$
is self-adjoint on $L^2(B,\pi)=L^2$ with spectrum composed of $N^2$ eigenvalues contained in $[-1,1]$.
 The largest eigenvalue of $M$ is  $\beta_0=1$, associated with constant eigenfunctions. Our focus will be to estimate the second largest eigenvalue, $\beta_1<1$ through the related quantity,
$$\lambda=1-\beta_1.$$
Following standard practice, we call $\lambda_1$ the spectral gap of the chain. It is the second lowest eigenvalue of $I-M$, which is minus the infinitesimal generator of the continuous semigroup of operators,
$$H_t: u\mapsto H_t u=e^{-t} \sum_{m=0}^\infty \frac{t^n}{n!}M^nu.$$
For each $\mathbf k\in B$, $H_t(\mathbf k,\cdot)$ is a probability distribution on $B$ and it is the distribution of the continuous time random Markov chains driven by the kernel $M$. It is well understood that
\begin{equation}\label{Ht}
\lim_{t\to +\infty} H_t(\mathbf k,\cdot)=\pi.
\end{equation}
Moreover, if we set $\pi_*=\min_B \{\pi\}$, we have (e.g., \cite[Cor. 2.1.5]{StF})
$$\max_{\mathbf k,\mathbf l}\left\{\left|\frac{H_t(\mathbf k,\mathbf l)}{\pi(\mathbf l)}-1\right|\right\}
\le \frac{1}{\pi_*} e^{-\lambda t}.$$
This basic inequality provides a simple quantitative control of the convergence of $H_t(\mathbf k,\cdot)$ to $\pi$
in terms of the spectral gap $\lambda$ and justify in large part our interest in estimating $\lambda$. A similar result
holds for the discrete time Markov chains but the details are complicated by the role played by the negative spectrum despite the fact that Metropolis chains typically have good aperiodicity properties.  In the sequel we focus on estimating $\lambda$ for a number of examples using the well established path technique (and variations on it) to obtain lower-bounds on $\lambda$.
In most cases, we establish also upper-bounds of the same order of magnitude. In the last section, we  briefly discuss what one can expect regarding the mixing time of the studied examples.

\section{The path technique for spectral gap lower-bounds}

\subsection{Dirichlet form notation} Let $B$ be a finite set and $K,\pi$ be be a reversible Markov kernel and its reversible probability measure on $B$.  These  assumptions mean that 
$$K: B\times B \to [0,1], \;(x,y)\mapsto K(x,y)$$
satisfies 
$$\sum_yK(x,y)=1 \mbox{ and } \pi(x)K(x,y)=\pi(y)K(y,x).$$ 
We can view $K$ as an operator acting on function defined on $B$ by the rule
$$Ku(x)=\sum_y K(x,y)u(y).$$
We work in the finite vector space of functions $u$ on $B$ equipped with the scalar product 
$$\langle u,v\rangle=\langle u,v\rangle_\pi=\sum_{b\in B} u(b)v(b)\pi(b).$$
Note that $$\pi(u^2)=\sum_{b\in B}|u(b)|^2 \pi(b)=\langle u,u\rangle.$$
On the this vector space, $K$ is diagonalizable with highest eigenvalue $1$ (associated with the constant functions) and second highest eigenvalue $\beta$ and spectral gap $\lambda=1-\beta$.
The Dirichlet form $\mathcal E=\mathcal E_{K,\pi}$ associated with the pair $(K,\pi)$ is the quadratic form
$$\mathcal E(u,v)= \sum_{x,y\in B} (u(x)-u(y))(v(x)-v(y)) \pi(x)K(x,y).$$
It allows us to compute the spectral gap $\lambda$ through the variational formula
$$\lambda=\inf \left\{\frac{\mathcal E(u,u)}{\pi(u^2)}: \pi(u)=0, \pi (u^2)\neq 0\right\}.$$
To obtain a lower-bound  $\lambda\ge \epsilon >0$, we need to prove that for all  functions $u$,
$$\mbox{Var}_\pi(u)=\pi(|u-\pi(u)|^2)$$
is bounded above by
$$\mbox{Var}_\pi(u)\le \frac{1}{\epsilon} \mathcal E(u,u).$$
To prove an  upper-bound  $\lambda\le \epsilon$, we need to find a function $u$ such that $\pi(u)=0$ and 
$$\frac{\mathcal E(u,u)}{\pi(u^2)}\le \epsilon.$$
The well-known formulas
$$\mbox{Var}_\pi(u)= \frac{1}{2}\sum_{x,y\in B}|u(x)-u(y)|^2\pi(x)\pi(y)=\min_{\xi\in \mathbb R}\sum_{b\in B}|u(b)-\xi|^2\pi(b)$$
are key for the path technique recalled  below.
\subsection{The path-technique} This work can be viewed as an illustration of the path techniques  we are about to describe. There are many variations that turn out to be useful when treating particular examples (we do not discuss here all possible variations).  It is useful to view $B$ as the vertex set of a graph whose edge set $E$
consists of those pairs $e=(x,y)$ in $B\times B$ such that $x\neq y$ (no loops) and $\pi(x)K(x,y)>0$ (possible connection through $K$). We will use the graph distance $d: B\times B\to [0,+\infty)$ which, for  $x,y\in B$ is defined as  the minimal number of edges in $E$ that one must cross to go from $x$ to $y$. For a subset $A\subseteq B$, $d(x,A)=\min\{d(x,y): y\in A\}$. With a slight abuse of notation, for any edge $e=(x,y)$ and $z\in B$, $d(z,e)=d(z,\{x,y\})$.
For all pairs $e=(x,y)$ of vertices (not just those in $E$), we set
$$Q(e)=\pi(x)K(x,y)$$ 
so that a pair $e$ is in $E$ if and only if $Q(e)>0$. Because of our standing assumptions, $Q((x,y))=Q((y,x))$ and thus $(x,y)\in E\Longleftrightarrow (y,x)\in E$. We let $\check{e}=(y,x)$ if $e=(x,y)$. With this notation, we have
\begin{eqnarray*}\mathcal E(u,v)&=& \frac{1}{2}\sum_{e\in B^2} (u(y)-u(x))(v(y)-v(x))Q(e)\\
&=&\frac{1}{2}\sum_{e\in E} (u(y)-u(x))(v(y)-v(x))Q(e).\end{eqnarray*}

By definition, a path is any finite string of vertices, $\gamma=(x_0,\dots, x_n)$, such that any two consecutive vertices in the path form a pair which belongs to $E$.

Let $A$ be a particular subset of $B$  and let $\pi_A=\pi(A)^{-1}\left.\pi\right|_A$, the normalized restriction of $\pi$ to $A$.  Observe that
$$\pi(|u-\pi_A(u|_A)|^2)= \pi(u^2)+\pi_A(u|_A)^2-2\pi(u)\pi_A(u|_A)$$
and
$$\sum_{x\in B,y\in A}|u(x)-u(y)|^2\pi(x)\pi_A(y)= \pi(u^2)+\pi_A(u|_A^2)-2\pi(u)\pi_A(u|_A).$$
By Jensen's inequality, $\pi_A(u|_A)^2\le \pi_A(u|_A^2)$, and thus
$$\pi(|u-\pi_A(u|_A)|^2)\le  \sum_{x\in B,y\in A}|u(x)-u(y)|^2\pi(x)\pi_A(y).$$ 
Note that, as noted earlier, when $A=B$, this can be improved to
$$\pi(|u-\pi(u)|^2)=\frac{1}{2}  \sum_{x,y\in B}|u(x)-u(y)|^2\pi(x)\pi(y).$$ 
Now, Let $\Gamma=\{\gamma_{xy}: x\in B,y\in A\}$ be a collection of paths $\gamma_{xy}$, indexed by the order pairs $xy$, $x,y\in B$, such that $\gamma_{xy}$ starts at $x$ and ends at $y$.   Let  $w: E\to (0,+\infty)$ be a positive edge function which we call a weight function.
For an edge $e=(a,b)$, write $du(e)=u(b)-u(a)$ for any function $u:B\to \mathbb R$.
Observe that (edges in a path are the pair of consecutive vertices along the path)
$$u(y)-u(x)= \sum_{e\in \gamma_{xy}} du(e)$$
and
$$|u(y)-u(x)|^2\le \left(\sum_{e\in \gamma_{xy}}\frac{1}{w(e)^2}\right)\left(\sum_{e\in \gamma_{xy}}|du(e) |^2w(e)^2\right).$$
Set $$|\gamma|=|\gamma|_w=\sum_{e\in \gamma}\frac{1}{w(e)^2}$$
and call this the $w$-length of $\gamma$.
Multiply by $\pi(x)\pi_A(y)$ and sum over  $B\times A$ to obtain
$$\sum_{x\in B,y\in A}|u(x)-u(y)|^2\pi(x)\pi_A(y)\le \sum _{x\in B,y\in A}\sum_{e\in \gamma_{xy}}|\gamma|_w\sum_{e\in \gamma_{xy}} |du(e)|^2 w(e)^2.$$
We rearrange this as follows.
\begin{eqnarray*}\lefteqn{\sum_{x\in B,y\in A}|u(x)-u(y)|^2\pi(x)\pi_A(y)}&&\\
&\le& \sum_{e\in E}\left(\frac{w(e)^2}{Q(e)} 
\sum _{x\in B,y\in A\atop {\gamma_{xy}\ni e}}|\gamma_{xy}|_w\sum_{e\in \gamma_{xy}}\pi(x)\pi_A(y)\right) |du(e)|^2Q(e).\end{eqnarray*}

This gives the following result.
\begin{prop} \label{prop-W}Referring to the notation introduce above and for any set $A\subset B$,
$$\lambda\ge 1/W  \;\;\; (\mbox{resp. } \lambda\ge 2/W \;\;\mbox{ if }\;  A=B),$$
where $$W=\max_{e\in E}\left\{\frac{w(e)^2}{Q(e)}\sum_{x\in B,y\in A \atop {\gamma_{xy}\ni e}} |\gamma_{xy}|_w\pi(x)\pi_A(y)\right\}.$$
\end{prop}
In the sequel, we will often use the notation (the set $A$ will be clear from the context)
$$W(e)=\frac{w(e)^2}{Q(e)}\sum_{x\in B,y\in A \atop {\gamma_{xy}\ni e}} |\gamma_{xy}|_w\pi(x)\pi_A(y).$$
The basic ideas used for this proposition are well-known but its seems difficult to locate a proper reference using both weights and the subset $A$ as we did above.

\begin{rem}[The use of symmetry when $A=B$] Assume that $A=B$, the path $\gamma_{xy}$ is the same as $\gamma_{yx}$ in reverse, and $w(\check{e})=w(e)$ for all $e$. This is a very common situation. Let $\bar{e}=\{x,y\}$ (non-oriented edge) and write  $\bar{e}\in \gamma_{xy}$ if either $e$ or $\check{e}$ is on $\gamma_{xy}$.  
 It then follows that $$W(\check{e})=W(e) =\frac{1}{2} \frac{w(e)^2}{Q(e)}\sum_{x,y\in B \atop {\gamma_{xy} \ni \bar{e} }} |\gamma_{xy}|_w\pi(x)\pi(y).$$
 Now, this symmetric expression allow us to break the symmetry via additional conditions on $x$ and $y$. For instance,
we can bound $W(e)$ from above by 
 $$ \frac{w(e)^2}{Q(e)}\sum_{x,y\in B, \pi(x)\le \pi(y) \atop {\gamma_{xy} \ni e }} |\gamma_{xy}|_w\pi(x)\pi(y)
  +  \frac{w(e)^2}{Q(e)}\sum_{x,y\in B, \pi(x)\le \pi(y) \atop {\gamma_{xy} \ni \check{e} }} |\gamma_{xy}|_w\pi(x)\pi(y)
$$
 or by 
 $$ \frac{w(e)^2}{Q(e)}\sum_{x,y\in B, d(x,e)\le d(y,e)  \atop {\gamma_{xy} \ni e }} |\gamma_{xy}|_w\pi(x)\pi(y)
 +  \frac{w(e)^2}{Q(e)}\sum_{x,y\in B, d(x,e)\le d(y,e)  \atop {\gamma_{xy} \ni \check{e} }} |\gamma_{xy}|_w\pi(x)\pi(y) .$$ This can be very useful in estimating $W$.  For instance, in this context,
\begin{equation}\label{symred}
W\le 2\max\left\{ \frac{w(e)^2}{Q(e)}\sum_{x,y\in B, \pi(x)\le \pi(y) \atop {\gamma_{xy} \ni e }} |\gamma_{xy}|_w\pi(x)\pi(y) \right\}.\end{equation}
 We will make use of this remark in the last section of this paper.
 \end{rem}

\subsection{Two $1$-dimensional families of examples} \label{sec-dim1}
Although our main goals is to illustrate the path technique in $2$ or more dimension, it is worth starting with two related one dimensional examples which are not found in the literature.
See \cite{What,Nash} for other $1$-dimensional examples.  

\subsubsection{First family} Our first family of examples comes from the choice 
$$\pi(x)=\pi_{a_-,a_+}(x)= Z_{a_-,a_+}\times \left\{\begin{array}{cl} (1+|x|)^{a_-} &\mbox{ if } x\le 0,\\
(1+|x|)^{a_+} &\mbox{ if } x> 0,\end{array}\right.  x\in B=\{-N,\dots,N\},$$
where $a_-,a_+>0$. It serves as a warm-up for the second family which will be discussed in Section \ref{fam2}.
Without loss of generality, we assume that
$$a_-\le a_+.$$ 
The case $a_-=a_+$ is treated in \cite{Nash}.
The proposal chain is the chain corresponding to a simple random walk with loops at the two ends, and the Metropolis kernel
 is given by $$Q(e)=\pi(x)M(x,y)=\frac{1}{2}\min\{\pi(x),\pi(y)\}$$ if  $|y-x|=1$ and $0$ if $|y-x|>1$ (the value when $x=y$ can be deduced from the given formula because $\sum_y M(x,y)=1$). 
 Also,  $$Z^{-1}=Z_{a_-,a_+}^{-1}=1+\sum_1^N (1+k)^{a_-}+\sum_1^N (1+k)^{a_+}$$ so that  $Z_{a_-,a_+} \asymp N^{-(1+a_+)}$.
 
 \begin{prop}\label{prop-asymvalley} For each fixed pair $0\le a_-\le a_+<+\infty$, the Metropolis chain for $\pi_{a_-,a_+}$ on $B=\{-N,\dots,N\}$ has spectral gap $\lambda$ bounded above and below by
 $$\lambda \asymp \left\{\begin{array}{cl} 1/N^{2}&\mbox{ if } a_-\in (0,1),\\
1/N^2\log N &\mbox{ if } a_-=1,\\
1/N^{1+a_-} &\mbox{ if } a_->1. \end{array}\right.$$ 
 \end{prop}
\begin{proof} 
 In one dimension, there is no questions about the choice of path from $x$ to $y$. The question that remains is the choice of weight and this choice is important here. In general, choosing weight is more of an art than a science and proceeds by trial and error.  In this case, we make the choice (recall that $a_-\le a_+$)
 $$w(x,x+1)=w(x+1,x)=(1+|x|)^{a_-/2}$$
and set $$D_{a_-}= \sum_0^N(1+k)^{-a_-}\asymp \left\{\begin{array}{cl} N^{1-a_-} &\mbox{ if } a_-\in (0,1),\\
\log N &\mbox{ if } a_-=1,\\
1 &\mbox{ if } a_->1. \end{array}\right.$$
The quantity $D_{a_-}$ controls the $w$-length of paths.
 
  We use Proposition \ref{prop-W} with $A=B$.  For a fixed edge $e=(b,b+1)$, we need to compute
$$ W(e)=2 \frac{w(e)^2}{Q(e)}\sum_{x\le  b<y} |\gamma_{x,y}|_w\pi(x)\pi(y).$$
If $ b<0$,
$$W(e )\asymp  D_{a_-}  (1+N)^{a_-}(N-|b|)$$
and, if $0\le b<N$,
\begin{eqnarray*}W(e)&\asymp &  Z (1+b)^{a_--a_+}  D_{a_-} (N^{1+a_-}+ (1+b)^{1+a_+}) (1+N)^{a_+}(N-b).\\
&\asymp&  D_{a_-} \left(\frac{N^{1+a_-}}{(1+b)^{a_+-a_-}}+(1+b)^{1+a_-}\right)(1-b/N).
\end{eqnarray*}
 Taking maximum over  $-N\le b<N$ gives
 $$W\asymp D_{a_-}(1+N)^{1+a_-}\asymp \left\{\begin{array}{cl} N^2&\mbox{ if } a_-\in (0,1),\\
N^2\log N &\mbox{ if } a_-=1,\\
N^{1+a_-} &\mbox{ if } a_->1. \end{array}\right.$$ 
Proposition \ref{prop-W} gives $\lambda\ge 2/W$ and. In fact, one can show that $\lambda\asymp 1/W$ in each of  these cases. We simply indicate which test function to use to obtain, in each case, an appropriate lower-bound on $\lambda$. In the first two cases, the constant $c_-$ and $c_+$ below are chosen so that $\sum_1^Nf(-k)\pi(k)=
\sum_1^N f(k)\pi(k)$, so that $\pi(f)=0$ and the definition of $f$ ensures that $c_-\asymp c_+$ as functions of $N$.
\begin{itemize}
\item ($a_-\in (0,1)$) 
$$f(k)=\left\{\begin{array}{cl} c_-|k| & \mbox{ if } \; k\in \{-N,\dots, 0\},\\ -c_+|k|^{1+a_--a_+} &\mbox{ if }\; k\in \{1,\dots,N\}.\end{array}\right.$$ 
By inspection, $\mathcal E(f,f) \asymp c_-^2ZN^{1+a_-} $ and $\mbox{Var}_{\pi}(f)\asymp c_-^2 ZN^{3+a_-}$.
\item ($a_-=1$) 
$$f(k)=\left\{\begin{array}{cl} c_- \sum_1^{|k|} \frac{1}{m} & \mbox{ if } \;k\in \{-N,\dots,0\},\\
c_+k^{ 1-a_+}\log k &\mbox{ if }\; k\in \{1,\dots,N\}.\end{array}\right.$$
By inspection, $\mathcal E(f,f)\asymp c_-^2 Z \log N$ and $\mbox{Var}_{\pi}(f)\asymp c_-^2Z N^2(\log N )^2$.
\item ($a_->1$)  $\xi= Z\left(\sum_0^N (1+k)^{a_-}\right) \asymp N^{a_--a_+}$ and 
$$f(k)=\left\{\begin{array}{cl}  1- \xi & \mbox{ if } \;k\in \{-N,\dots,0\},\\
-\xi&\mbox{ if }\; k\in \{1,\dots,N\}.\end{array}\right.$$
By inspection, $\pi(f)=0$, $\mathcal E(f,f)\asymp  Z$ and $\mbox{Var}_{\pi}(f)\asymp ZN^{a_-+1}$.
\end{itemize}
\end{proof}

\subsubsection{Second family} \label{fam2}Our second family of examples extends the  first by allowing the two sides of the interval around $0$ to have different length, $N_-, N_+$. The two exponents $a_-,a_+\ge 0$ are not ordered.  Set
$$B=\{-N_-,\dots,0,\dots,N_+\}$$
with $$\pi(x)=Z\times\left\{\begin{array}{cl} (1+|x|)^{a_-}&\mbox{ if } x\in \{-N_-,\dots,0\},\\
(1+|x|)^{a_+} &\mbox{ if } x\in \{1,\dots,N_+\}.\end{array}\right.$$ 
 What is the order of magnitude of the spectral gap $\lambda$ for the associated Metropolis chain?
 
 First observe that $$Z^{-1}=1+\sum_1^{N_-} (1+k)^{a_-}+\sum_1^{N_+}(1+k)^{a_+}\asymp (1+N_-)^{a_-+1}+(1+N_+)^{a_++1}.$$
 
We start by deriving the following upper-bound.
\begin{lem} \label{lem-1}
For fixed reals  $a_-,a_+\ge 0$ there is a constant $C$ such that for
any integers $N_-,N_+\ge 0$, the spectral gap of the Metropolis chain for the measure $\pi$ above on $\{-N_-,\dots,N_+\}$ satisfies
$$\lambda\le C\min\left\{ \frac{1}{(N_-+N_+)^2}, \max\left\{\frac{1}{(1+N_-)^{1+a_-}},\frac{1}{(1+N_+)^{1+a_+}}\right\}\right\}.$$
\end{lem}
\begin{proof} We prove separately that $\lambda$ is bounded above (up to a multiplicative constant)  by each of the quantities  
$\max\left\{\frac{1}{(1+N_-)^{1+a_-}},\frac{1}{(1+N_+)^{1+a_+}}\right\}$ and $\frac{1}{(N_-+N_+)^2}$.

For the first part of the proof, without loss of generality, we assume that 
$$(1+N_-)^{1+a_-}\le (1+N_+)^{1+a_+}.$$ We then want to show that $\lambda\le C(1+ N_-)^{ -(1+a_-)}$.
We use the test function $f$ equals to  $1-\xi$ on $\{-N_-,\dots,0\}$, and  $-\xi $ on $\{1,\dots,N_+\}$. Here 
$$\xi= \pi(\{-N_-,\dots,0\})=\frac{\sum_0^{N_-}(1+k)^{a_-}}{ \sum_0^{N_-}(1+k)^{a_-}+\sum_1^{N_+} (1+k)^{a_+}}
\asymp \frac{(1+N_-)^{1+a_-}}{(1+N_+)^{1+a_+}}$$ so that $\pi(f)=0$.  We simply need to estimate $\mathcal E(f,f)$ from above and $\mbox{Var}_\pi(f)=\pi(f^2)$ from below.  Only the edges $\{(0,1),(1,0)\}$
contribute to $\mathcal E(f,f)\asymp Z$ and 
$$\pi(f^2)\asymp Z \left((1-\xi)^2 (1+N_-)^{1+a_-}+\xi^2(1+ N_+)^{1+a_+}\right)\asymp  Z(1+N_-)^{1+a_-}.$$
This proves that  $\lambda \le C(1+N_-)^{-(1+a_-)}$. Removing the extra assumption that $(1+N_-)^{1+a_-}\le (1+N_+)^{a_+}$,   we obtain
$$\lambda\le C\max\left\{\frac{1}{(1+N_-)^{1+a_-}},\frac{1}{(1+N_+)^{1+a_+}}\right\}.$$

Our next goal is to show that $\lambda\le C(N_-+N_+)^{-2}$. For this, we can assume without loss of generality that
$N_-\le N_+$ so that the desired result reads $\lambda\le C' N^{-2}_+$.  We consider two cases, depending on which side gives the main contribution to $Z^{-1}$.  

We start with the case when $N_-^{1+a_-}\le N_+^{1+a_+}$
and use the test function $f(k)=k$ whose Dirichlet form satisfies  $\mathcal E(f,f)\asymp 1$. The mean of $f$, $\pi(f)$, is a number between 
$-N_-$ and $N_+$ and, whatever it is, there is always an interval  of size $N_+/8$ in $\{[N_+/2],\dots, N_+\}$
over which $|f-\pi(f)|\ge N_+/8$.  This implies that $\mbox{Var}_\pi(f)\ge cN_+^2$ and thus $\lambda\le C'N_+^{-2}$ as desired. This does not work if  $N_-^{1+a_-}$ is the dominant term in $Z^{-1}\asymp N_-^{1+a_-}+N_+^{1+a_+}$.

When $N_-^{1+a_-}>N_+^{1+a_+}$, consider the test function 
$$f(k)=\left\{\begin{array}{cl} -N_+^{1+a_+}k&\mbox{ if } k\in \{-N_-,\dots,0\},\\
cN_-^{1+a_-}k &\mbox{ if } k\in \{1,2,\dots,N_+\}.\end{array}\right.$$ 
The constant $c\asymp 1$ is chosen to that $\pi(f)=0$.
The Dirichlet form of $f$ satisfies  
$$\mathcal E(f,f)\asymp Z(N_+^{2(1+a_+)}N_-^{1+a_-}+N_-^{2(1+a_-)}N_+^{1+a_+})\le 2ZN_-^{2(1+a_-)}N_+^{1+a_+}$$
because we assume  that  $N_-^{1+a_-}>N_+^{1+a_+}$. The variance of $f$ satisfies
$$\mbox{Var}_\pi(f)\asymp Z(N_+^{2(1+a_+)}N_-^{3+a_-}+N_-^{2(1+a_-)}N_+^{3+a_+})\ge ZN_-^{2(1+a_-)}N_+^{3+a_+}.$$  
This gives us $\lambda\le C'N_+^{-2}$ as desired.
\end{proof}

Next, we derive lower-bounds on the spectral gap.
\begin{lem}\label{lem-2} ~

\begin{itemize}
\item Assume that
$a_-,a_+\in (1,+\infty) $. Then 
$$\lambda\asymp \min\left\{\frac{1}{(N_-+N_+)^2},  \max\left\{ \frac{1}{(1+N_-)^{1+a_-}},\frac{1}{(1+N_+)^{1+a_+}}\right\}\right\}.$$
\item Assume that $\min\{a_-,a_+\}\in (0,1)$. Then
$$\lambda\asymp \frac{1}{(N_-+N_+)^2}.$$
\item Assume that $a_+=a_-=1$ and $N_-\le N_+$. Then 
$$\lambda\asymp   \frac{ 1}{N_+^2+N_-^2\log N_- .}$$
In particular $\lambda\asymp N_+^{-2}$ whenever $N_-\le N_+/\sqrt{\log N_+}$.
\item Assume that  $\min\{a_-,a_+\}=1$ and $\max\{a_-,a_+\}>1$. For convenience,
 assume $a_-=1<a_+$.  Then
 $$\lambda\asymp \max\left\{ \frac{1}{N_+^2+N_-^2\log N_-},\frac{1}{N_-^2 +N_+^{1+a_+}\log N_+}\right\}.$$
 \end{itemize} \end{lem}
 \begin{rem} The last two cases can be stated together as follows: if  $\min\{a_-,a_+\}=1$ then
 $$\lambda\asymp   \max\left\{ \frac{1}{N_+^2+N_-^{1+a_-}\log N_-},\frac{1}{N_-^2 +N_+^{1+a_+}\log N_+}\right\}.$$ 
 \end{rem}
\begin{proof} Because of Lemma \ref{lem-1}, in the first two cases, it suffices to prove the lower-bounds. 
In the first case when $a_-,a_+\in (1,+\infty)$, without loss of generality, we can assume  that $N_-^{1+a_-}\le N_+^{1+a_+}$ so that we want to prove that $$\lambda\ge c\min\{(1+ N_+)^{-2},(1+N_-)^{-(1+a_-)}\}.$$  Note that $Z\asymp N_+^{-(1+a_+)}$ and fix $1<\eta<\min\{a_-,a_+\}$.
For  an edge $e$ whose vertex farthest from $0$ is $k$, set
$w(e)=(1+|k|)^{\eta/2}$ when $|k|\le N_-$  and $w(e)= N_+^{1/2}$ if $N_-<k\le N_+$ (if such $k$ exist). This gives $|\gamma_{xy}|_w\asymp 1$ for all $x,y\in \{-N_-,\dots,N_+\}$. We consider
$$ W(e)= \frac{w(e)^2}{Q(e)}\sum_{x\le  b<y} |\gamma_{x,y}|_w\pi(x)\pi(y)$$ and write (again, $k$ denote the vertex of $e$ farthest from $0$)
$$W(e)\le CZ\left\{\begin{array}{l} N_+ k^{-a_+} (N_+-k+1)N_+^{a_+}(N_-^{1+a_-}+k^{1+a_+})\mbox{ if }  N_-<k\le N_+,\\k^{\eta-a_+} (N_+-k+1)N_+^{a_+}(N_-^{1+a_-}+k^{1+a_+})\mbox{ if } 0<k\le N_-,\\
|k|^{\eta-a_-} (N_--|k|+1)N_-^{a_-}(N_+^{1+a_+}+|k|^{1+a_-})\mbox{ if }  k<0. \end{array}\right.$$
Because, by assumption,  $Z\asymp N_+^{-1-a_+}$, this gives
$$W(e)\le C\left\{\begin{array}{l}N_+ k^{-a_+} (N_-^{1+a_-}+k^{1+a_+})\mbox{ if }  N_-<k\le N_+ ,\\k^{\eta-a_+}(N_-^{1+a_-}+k^{1+a_+})\mbox{ if } 0<k\le N_-,\\
|k|^{\eta-a_-} N_-^{1+a_-} \mbox{ if }  k<0. \end{array}\right.$$
By inspection, $W(e)\le C \max\{N_+^{2}, N_-^{1+a_-}\}$ where $C$ may have changed from the previous line.
Under our assumption, this gives the desired lower-bound on $\lambda$.
         
Next, we treat the case when $\min\{a_-,a_+\}\in (0,1)$. Without loss of generality we can assume that $a_-\le a_+$.  We have $Z^{-1}\asymp N_-^{1+a_-}+N_+^{1+a_+}$. 
We set $$w(k-1,k)=w(-k+1,-k)=(1+k)^{a_-/2},\;\;k>0.$$ This gives $|\gamma_{xy}|_w\le C(N_-+N_+)^{1-a_-}=D$ and
$$W(e)\le CD N_-^{1+a_-} $$ if $e=(-k+1,-k)$, $k>0$, and $$W(e)\le CD |k|^{a_--a_+} (N_-^{1+a_-}+ |k|^{1+a_+})$$
if $e=(k-1,k)$, $k>0$.  The maximum for $k>0$ is less than $CD(N_-+ N_+)^{1+a_-}$. All together, this yields
$W(e)\le C(N_-+N_+)^2$ as desired.

Next, assume that $a_-=a_+=1$ and $N_-\le  N_+$. In this case, we need to prove both an upper-bound and a lower-bound. We start with the lower-bound. For an edge  $e$ whose vertex furthest to $0$ is $k$, set
$$w(e)=\left\{\begin{array}{cl} |k|^{1/2} &\mbox{ if } k\in [-N_-, N_-],\\(N_+/\log N_-)^{1/2}
& \mbox{ if } k>N_-.\end{array}\right.$$
The $w$-length  $|\gamma_{xy}|_w$ of any path is bounded above by  
$$D=(3-N_-/N_+)( \log N_-)\le 2\log N_-.$$
The quantity $W(e)$  is bounded above by
$$W(e)\le C D\left\{\begin{array}{cl}  N_-^2 &\mbox{ if }  k\in [-N_-, N_+],\\
\frac{N_+}{k \log N_-}(N_-^2+k^2)
& \mbox{ if } k\in\{N_-,\dots,N_+\}.\end{array}\right.$$
It follows that 
$$W(e)\le C\max\{ N_-^2 \log N_-, N_+^2\}.$$

To obtain a matching upper-bound, we use the test function
$$f(k)=\left\{\begin{array}{cl} -(N_+^2\log N_+ )\sum_1^{|k|}\frac{1}{m} &\mbox{ if } k\in\{-N_-,\dots,0\},\\
c(N_-^2\log N_- )\sum_1^k\frac{1}{m} &\mbox{ if } k\in\{1,\dots,N_+\}, \end{array}\right.$$
where $c\asymp 1$ is chosen so that $\pi(f)=0$. For this test function,
$$\mathcal E(f,f)\asymp N_+^{-2}[(N_+^2\log N_+)^2 \log N_- +(N_-^2\log N_-)^2\log N_+]$$
and
$$\mbox{Var}_\pi(f)\asymp N_+^{-2} (N_+\log N_+)^2(N_-\log N_-)^2(N_+^2+N_-^2).$$
Because we assume $N_-\le N_+$, this gives
$$\lambda \le  \frac{C}{N_-^2\log N_-}.$$
We already know that $\lambda\le CN_+^{-2}$ and the stated result follows.

Finally, we consider the case when $a_-=1<a_+$. First, assume that $N_+^{1+a_+}\ge N_-^2$. We start with the lower-bound.   For  an edge $e$ whose vertex farthest from $0$ is $k$, we set
$$w(e)=\left\{\begin{array}{cl}  |k|^{1/2} &\mbox{ if } k\in \{-N_-,\dots,-1\},\\
k^{a_+/2}(\log N_-)^{-1/2} &\mbox{ if } k\in \{1,\dots, [N_-^{2/(1+a_+)}]\},\\
N_+^{1/2}(\log N_-)^{-1/2} &\mbox{ if } k\in \{[N_-^{2/(1+a_+)}],\dots, N_+\} .\end{array}\right.$$
These edge weights give a maximal path length of order $\log N_-$ and the constant $Z^{-1}$ is or order $N_+^{1+a_+}$. We estimate 
$$W(e)\le C\times \left\{\begin{array}{cl} N_-^2(\log N_-) &\mbox{ if } k\in \{-N_-,\dots,-1\},\\
N_-^2& \mbox{ if } k\in \{1,\dots, [N_-^{2/(1+a_+)}]\},\\
N_+(k^{a_+})^{-1}(N_-^2+k^{1+a_+}) &\mbox{ if } k\ge[N_-^{2/(1+a_+)}] .\end{array}\right.$$
This gives   $W(e)\le C(N_+^2+N_-^2\log N_-)$ and 
$$\lambda\ge \frac{c}{N_+^2+N_-^2\log N_-}.$$
An upper-bound is needed only when $N_-\ge N_+$. Because we also assume $N_-^2\le N_+^{1+a_+}$, we have $\log N_-\asymp \log N_+$ in this case. For the upper-bound we use the test function
$$f(k)=\left\{\begin{array}{cl} -N_+^{1+a_+}\sum_1^{|k|}1/m &\mbox{ if } k<0,\\
cN_-^{2}\sum_1^{|k|}1/m &\mbox{ if } k>0.\end{array}\right.$$
The constant $c\asymp 1$ is chosen so that $\pi(f)=0$. We have
$$\mathcal E(f,f)\asymp N_+^{-(1+a_+)}(N_+^{2(1+a_+)}\log N_- +N_-^4\log N_+)$$ and
$$\mbox{Var}_\pi(f)\asymp N_+^{-(1+a_+)}(N_+^{2(1+a_+)}N_-^2(\log N_- )^2+N_-^4 N_+^{1+a_+}(\log N_+)^2).$$
This gives
$$\lambda\le  C\frac{N_+^{2(1+a_+)}\log N_-}{N_+^{2(1+a_+)} N_-^2(\log N_-)^2}\asymp \frac{1}{N_-^2 \log N_-}.$$

We are left with the case when $N_+^{1+a_+}\le N_-^2$ which implies $N_+\le N_-$ because $a_+> 1$.
We pick $$w(e)=\left\{\begin{array}{cl} ( k^{a_+}/\log N_+)^{1/2} &\mbox{ if } k\in \{1,\dots,N_+\},\\
|k|^{1/2} &\mbox{ if } k\in \{-N_+,\dots, -1\},\\
N_-^{1/2}/(\log N_+)^{1/2} &\mbox{ if } k\in \{-N_-,\dots, -N+\}.\end{array}\right.$$
This gives  a maximal $w$-length for paths of order $\log N_+$ 
and we get
$$W(e)\le C(\log N_+) \times \left\{\begin{array}{cl} N_+^{1+a_+}/(\log N_+)&\mbox{ if } k\in \{1,\dots,N_+\},\\
 N_+^{1+a_+}  & \mbox{ if } k\in \{- N_+,\dots,-1\},\\
 N_-/(|k|\log N_+)(N_+^{1+a_+}+|k|^2)  & \mbox{ if } k\in \{- N_-,\dots,-N_+\}.
 \end{array}\right.$$
 It follows that  $W(e)\le C\max\{N_-^2, N_+^{1+a+}\log N_+\}$ and
 $$\lambda\ge \frac{c}{N_-^2+N_+^{1+a_+}\log N_+}.$$
 
Given that $N_+\le N_-$, we already have the upper-bound $\lambda\le \frac{C}{N_-^2}.$ It thus suffices to consider the case when. $N^{1+a_+}\le N_-^2\le N_+^{1+a_+}\log N_+$. This implies that $\log N_+\asymp \log N_-$. We use the test function
$$f(k)=\left\{\begin{array}{cl} -N_+^{1+a_+}\sum_1^{|k|}1/m &\mbox{ if } k<0,\\
cN_-^{2}\sum_1^{|k|}1/m &\mbox{ if } k>0.\end{array}\right.$$
The constant $c\asymp 1$ is chosen so that $\pi(f)=0$. We have
$$\mathcal E(f,f)\asymp N_-^{-2}(N_+^{2(1+a_+)}\log N_- +N_-^4\log N_+)$$ and
$$\mbox{Var}_\pi(f)\asymp N_-^{-2}(N_+^{2(1+a_+)}N_-^2(\log N_- )^2+N_-^4 N_+^{1+a_+}(\log N_+)^2)$$
This gives
$$\lambda\le  C\frac{N_-^{4}\log N_+}{N_-^4N_+^{1+a_+} (\log N_+)^2}\asymp \frac{1}{N_+^{1+a_+} \log N_+}.$$
\end{proof}

\section{Some basic examples in higher dimensions} 

The same basic techniques discussed above apply in any dimension but many complications arise. In particular, what choice of paths is appropriate becomes a  nontrivial question. Overall, we understand much less in two than in one dimension. We hope the examples treated below provide some illustration. We set these examples in the context  explained in Section \ref{setup}, starting from a target density $f$ over $[0,1]^n$, most of the time with $n=2$ for simplicity.

\subsection{Exponential of a linear function} \label{sec-exp1}
This family of examples is based on the function $f(x)=e^{ax+by}$, $a,b\in \mathbb R$ with $|a|+|b|>0$
(the case $a=b=0$ is not particularly interesting).  We treat it as a warm-up for the next example and for illustration. It can be studied in many different ways.
Recall that we consider the grid approximation, $B_N=\{1,\dots,N\}^2$, of the unit square
 equipped with the probability measure
$$\pi(x,y)= Z_N^{-1} F_N(x,y), \;F_N(x,y)= \exp\left( \frac{1}{N}(ax+by-(a+b)/2)\right).$$
Note that, if $Q_{x,y}=\left\{(u,v): \left|u-\frac{x-\frac{1}{2}}{N}\right|<\frac{1}{N}, \left|v-\frac{y-\frac{1}{2}}{N}\right|<\frac{1}{N}\right\}$,
$$\left|F_N(x,y)- N^2\int_{Q_{x,y}}f(u,v)dudv\right|\le  e^{\frac{|a|+|b|}{N}}\frac{|a|+|b|}{N} F_N(x,y).$$
Hence, it is natural to assume that $(|a|+|b|)/N\le 1$.

The Metropolis kernel $M(\mathbf k,\mathbf l)$ on $B=B_N$ satisfies$$Q(\mathbf k,\mathbf l )\asymp \min\{\pi(\mathbf k),\pi(\mathbf l)\}\asymp \pi(\mathbf k)$$
 for all $\mathbf k,\mathbf l$ such that
$|k_1-l_1|+|k_2-l_2|=1$. There are two cases to consider.  Assume first that $ab=0$. By symmetry there is no loss of generality to assume $b=0$. We do not treat this case in detail but note that it can be analyzed by comparison
with  the product of two one-dimensional chains, the uniform proposal chain itself in the $y$ coordinate and the
one-dimensional Metropolis chain for $e^{ax/N}$ in the $x$ direction. See \cite[Sect. 6.1-6.2]{What} and \cite{DSCcompR}.  This comparison  results in a spectral gap of order $1/N^2$ for the two-dimensional Metropolis chain with $b=0$. In the case when either one or both of $a$ and $b$ are small, one can also show (e.g., by comparison) that the spectral gap is of order $1/N^2$. When both $a$ and $b$ are of order $N$, we will show below that the spectral gap is of order $1$.

This illustrates a crucial fact regarding the procedure we describe which is based on a uniform square grid: the size $N$ of grid has to be taken large enough that the grid  provides a good approximation and not so large that it flattens the profile of the function being approximated to the point of making its special features disappear. In the present example, the special feature is the exponential nature of the function.   If the coefficients $a,b$ in this example have different order of magnitude, the correct choice would be to use a rectangular discretization that would take this into account.

We now treat the case when $\min\{|a|,|b|\} \geq \epsilon N$ where $\epsilon>0$ is fixed. Under this assumption $F_N$ has a unique local maximum which is also a global maximum and is located at one of the four corners, call it $\mathbf k_*$. Moreover, $\pi(\mathbf k_*)\asymp_\epsilon 1$.  Since it is clear by reason of symmetry that we can assume $a,b>0$, we will do so in all arguments below.  When $a,b>0$, $\mathbf k_*=(N,N)$, 
and 
\begin{eqnarray*}Z^{-1}&=&\sum_1^N\sum_1^N \exp\left( \frac{1}{N}(ax+by-(a+b)/2)\right)\\
&=&
e^{-\frac{a+b}{2N}}\left(\sum_{x=1}^N\sum_1^N e^{\frac{a}{N}x}\right)\left(\sum_1^N\sum_{y=1}^N e^{\frac{b}{N}y}\right)\\
&=&e^{\frac{a+b}{N}} e^{a+b} \frac{1-e^{-a(1+1/N)}}{1-e^{-a/N}} \frac{1-e^{-b(1+1/N)}}{1-e^{-b/N}}
\asymp_\epsilon  e^{a+b} \end{eqnarray*} and
$$\pi(\mathbf k_*)= Z e^{a+b}\asymp_\epsilon 1.$$
\begin{prop} Fix $\epsilon>0$. When $f(x,y)=e^{ax+by}$ with $|a|+|b|\le N$ and $\min\{|a|,|b|\}\ge \epsilon N$, the spectral gap $\lambda$ of the associated 
finite Metropolis chain on $B_N$ satisfies
$\lambda\asymp_\epsilon 1$
  \end{prop}
\begin{proof} The only thing to prove is the lower-bound and this can be done by using
Proposition \ref{prop-W} with $A=\{\mathbf k_*\}$ and weights $w(e)=Q(e)^\theta$ where $\theta<1/2$. For each $\mathbf k\in B$, we define the path $\gamma_{\mathbf k}$ to be  one of the (possibly  two) paths from $\mathbf k$ to $\mathbf k_*$
with only one turn and along which $\pi$ increases at each step. Without loss of generality, we assume $a,b>0$ so that  $\mathbf k_*=(N,N)$, and we  let the path $\gamma_{\mathbf k}$ from $\mathbf k$ to $\mathbf k_*$ be the path that go up to the top then right. 
The $w$-length $|\gamma_{\mathbf k}|_w$ satisfies 
$$|\gamma_{\mathbf k}|_w\asymp _\epsilon  \pi(\mathbf k)^{-2\theta}.$$
Fix an edge $e$ and let $\mathbf n$ be the vertex on $e$ farthest from $\mathbf k_*$ along the path. We have
\begin{eqnarray*}W(e)&=& \frac{w(e)^2}{Q(e)}\sum_{\mathbf k\in B, \gamma_{\mathbf k}\ni e} |\gamma_{\mathbf k}|_w\pi(\mathbf k)\\
&\asymp _\epsilon & Q(e)^{2\theta-1} \sum_{\mathbf k\in B, \gamma_{\mathbf k}\ni e} \pi(\mathbf k)^{1-2\theta}\\
&\asymp_\epsilon  &Q(e)^{2\theta-1} \pi(\mathbf n)^{1-2\theta} \asymp_\epsilon 1 .\end{eqnarray*}\end{proof}

\subsection{Further examples with exponential fall-off} \label{sec-exp2}
In this  subsection, we assume that
$f$ as the form $f(x)= e^{-g(x)}$ where $g$ has the following  crucial properties. The constants $a,A$ introduced in these properties are key parameters.
\begin{enumerate}
\item[(1)] There exists $A>0$ such that, for all $z=(x,y),z'=(x',y')\in [0,1]^2$, $|g(z)-g(z')|\le A\|z-z'\|$;
\item[(2)] There exists $0<a\le A$ such that, for some $z_0\in [0,1]^2$ and all $z\in [0,1]^2$ and $t\in [0,1]$, 
$$g(z_0)=0 \;\mbox{ and }\; g(z)-g(z+t(z_0-z)) \ge a t ;$$
\item[(3)] The ratio $A/a$ is bounded above by a fixed constant $C\ge 1$.
\end{enumerate}
So, essentially, $f$ has a maximum at $z_0$ with exponential fall-off. The condition $g(z_0)=0$ is not restrictive, it is a simple normalization.   Condition (1)
is a regularity and growth condition ($g$ is Lipschitz with constant $A$). The main part of condition (2) ensures a uniform exponential growth along straight lines ending at $z_0$ with rate $a$. It is clear from (1)-(2) that $0<a\le A<+\infty$. Condition (3) is crucial and ensures that the key parameters $a$ and $A$ 
are comparable in the sense that $A/a\asymp_C1$.

The discretization on the size $N$ grid  is the function
$$F_N(x,y)= e^{-g(((x-1/2)/N,(y-1/2)/N))},\;\;(x,y)\in B_N=\{1,\dots,N\}^2$$
and, in order, to have a good approximation, we need to choose $N$ of order at least $A$. In order to retain the exponential fall-off, we also need to have $N$ not much bigger than $a$. This means the we need to choose  $ \epsilon A\le N \le \epsilon^{-1} a$, say, and this is possible because $A/a\le C$. 
The positive real number $\epsilon$ can be taken of order $1/\sqrt{C}$ which we assume in what follows.

\begin{prop} \label{pro-exp2}Fix $C>0$. Let $f(z)=e^{-g(z)}$ with  $g$ having the properties {\em (1)-(2)-(3)} stated above with constant $C$  and key parameters $a,A$.  The spectral gap $\lambda$ of the associated  finite Metropolis chain on $B_N$  with $ N\asymp_C A\asymp_C a$ satisfies
$\lambda\asymp_{C} 1 .$
  \end{prop}
\begin{proof} Let $z^N_0=(x_0^N,y_0^N)$ be the point of $B_N$ closest to $Nz_0=(Nx_0,Ny_0)$ and observe that $F_N(z_0^N)\approx_C 1$ because of properties (1)-(2) and the choice of $N$.  First, we show that $$Z_N^{-1}=\sum_{z\in B_N} F_N(z) \asymp _C F_N(z^N_0).$$  So it suffices to prove that $\sum_{z\in B_N} F_N(z)\le C_1$, independently of $a,A,N$ and the function $g$.  This follows easily from the fact that
$$F_N(z) =e^{-g( z')},\;\;z=(x,y)\in B_N,\;\;z'=(x',y')\in [0,1]^2,$$
where
\begin{equation} \label{eqz'}\left\{\begin{array}{c} x'=(x-1/2)/N,\\
y'=(y-1/2)/N.\end{array}\right.\end{equation}
Set $z'_0= (x^N_0/N,y^N_0/N)$ and write
 \begin{eqnarray*}g(z') &\ge &  a\|z'-z_0\|\\
 &\ge & a(\|z'-(z_0^N/N)\|-\|(z_0^N/N)-z_0\|)\\
&\ge & -\sqrt{2}\frac{a}{N}  + \frac{a}{N} \|z-z^N_0\|\\
&\ge & -\sqrt{2}\epsilon^{-1} + \epsilon \|z-z^N_0\|.\end{eqnarray*} 
It follows that
$$\sum _{B_N} F_N(z)\le e^{\sqrt{2}\epsilon^{-1}} \sum_{B_N} e^{-\epsilon \|z-z_0^N\|} \le C_2(\epsilon)$$
as desired.

Having proved that $Z_N^{-1}\asymp 1\asymp F_N(z^N_0)$, we can use Proposition \ref{prop-W} with the set $A=\{z^N_0\}\subset B=B_N$ (the set $A$ in Proposition \ref{prop-W} has, of course, nothing to do with the constants $a,A$ that define the key property of the function $g$).  
For each $z\in B_N$, we need a path $\gamma_z$ from $z$ to $z_0$ and we pick a discrete path that remains at distance at most $\sqrt{2}$ from the straight line from $z$ to $z^N_0$.  For each edge $e$, we pick the weight $w(e)=Q(e)^{\theta}$ with $\theta\in (0,1/2)$, e.g., $\theta=1/4$.  Here 
$Q(e)=\frac{1}{2}\min\left\{ \pi_N(e_-),\pi_N(e_+)\right\}$, $e=(e_-,e_+)$ (for some arbitrary orientation of the edge of the square grid, say, up and to the right).
By assumption (1) and the choice of $N$,   we have
$$  Q(e)\asymp _C  \pi_N(e_-)\asymp _C  \pi_N(e_+).$$
Observe that assumption (2) and the choice of $N$ show that the maximal $w$-length of one of our chose  paths $\gamma_z$ satisfies
$$|\gamma_z|_w=\sum_{e\in \gamma_z}
 Q(e)^{-2\theta} \le C_1\sum_{e\in \gamma_z} \pi(e_ -)^{-2\theta}\le C_1\pi(z)^{-2\theta}  \sum_{e\in \gamma_z} \left(\frac{\pi(z)}{\pi(e_-)}\right)^{2\theta}.$$
We have
$$\frac{\pi(z)}{\pi(e_-)}=e^{-(g(z')-g(e'_-))}$$ 
and, because $e'_-$ is at distance at most $\sqrt{2}/N$ of the straight line from $z$ to $Nz_0$
$$g(z') -g(e'_-)\ge  -C_2 (A/N) +a\|z-z_0\|/N.$$
If follows that
$$|\gamma_z|_w\le C_3(\epsilon) \pi(z)^{-2\theta}.$$
Finally, for any fixed edge $e$, we need to estimate
\begin{eqnarray*}W(e)&=& \frac{w(e)^2}{Q(e)}\sum_{z \in B_N, \gamma_{z}\ni e} |\gamma_z|_w\pi(z)\\
&\asymp _\epsilon & Q(e)^{2\theta-1} \sum_{z\in B_N, \gamma_{z}\ni e} \pi(z)^{1-2\theta} .\end{eqnarray*}
We claim that $$\sum_{z\in B_N, \gamma_{z}\ni e} \pi(z)^{1-2\theta} \le C_4(\epsilon) \pi(e_-)^{1-2\theta}.$$ 
Indeed, because $e_-$ is close to the straight line from $z$ to $Nz_0$, for any  integer $\ell$, there are at most order 
$\ell $ point $z$ such that  $\|z-e_-\|\asymp \ell$.
Noting that $1-2\theta>0$, it follows that 
\begin{eqnarray*}\sum_{z\in B, \gamma_{z}\ni e} \pi(z)^{1-2\theta}&\le & \pi(e_-)^{1-2\theta}
\sum_{z\in B_N, \gamma_{z}\ni e} \left(\frac{\pi(z)}{\pi(e_-)}\right)^{1-2\theta}\\
&\le & C_5(\epsilon)  \pi(e_-)^{1-2\theta} \sum \ell e^{-C_6a\ell /N} \\
&\le & C_7(\epsilon)  \pi(e_-)^{1-2\theta}. \end{eqnarray*}
This is the desired inequality. It gives $ W(e)\le C_7(\epsilon)$  and Proposition \ref{pro-exp2} follows.
\end{proof}

\subsection{No worse than the flat spectral gap $1/N^2$}
In \cite[Proposition 6.3]{What}, it is proved that for the Metropolis chain based on  nearest neighbor random walk on $\{1,\dots,N\}$ with target a distribution $\pi$ that has a unique local maximum, the spectral gap is greater or equal to $1/2N^2$. We do not know of a similar result in higher dimension and it seems to be difficult to capture different cases for which the spectral gap is at least or order $1/N^2$ with a single argument.
The following proposition captures one particular class of examples. 

In this subsection, we explicitly work in a fixed dimension $n\ge 2$. For $A\ge 1$ and $\epsilon,\eta>0$, let $\mathcal C_{n,\epsilon,\eta}(A)$ be the class of continuously differentiable function $f$ on $[0,1]^n$ such that
\begin{equation}\label{C1}
\min_{[0,1]^n}\{f\}>0,\;\;\sup_{[0,1]^n, i=1,\dots,n}\{ |\partial_i \log f|\}\le A,\end{equation}
\begin{equation} \label{C2}\forall\,x,y\in [0,1]^n, t\in [0,1],\;\; f(x+t(y-x))\ge \epsilon \min\{f(x),f(y)\},\end{equation}
and
\begin{equation}\label{C3}\eta \|f\|_\infty\le \int_{[0,1]^n} f(x) dx.\end{equation}
As in the previous example, the parameter $A$ is a key parameter and one can think of it as being large (otherwise the result is relatively obvious).
The second hypothesis is a weak form of convexity of the level sets of $f$. If the sets $\{x: f(x)>t\}$ were convex then, certainly,
we would have
 $$\forall\,x,y\in [0,1]^n, t\in [0,1],\;\; f(x+t(y-x))\ge \ \min\{f(x),f(y)\}$$
 and we could take $\epsilon =1$. 
 
 It may be useful to observe that the function $f(x,y)=e^{A(x+y)}$ does not belong to $\mathcal C_{2,\epsilon,\eta}(A)$ because (\ref{C3}) fails but, for any 
 $\theta>0$, $f(x,y)= (1+A(x+y))^{\theta}$ does belong to $\mathcal C_{2,1,\eta_\theta}(A\theta )$ for some easily computed explicit $\eta_\theta$.  Other examples satisfying (\ref{C1})-(\ref{C2})-(\ref{C3}) are $f(x,y)=(1+ ((Ax)^2+ (Ay)^4))^{-\theta}$, $\theta>0$, and $f(x,y)= e^{A\min\{ x+y, 1\}}.$
 
 Next, as always, we pick an integer $N\ge A$ to ensure a good approximation between the discretized function
 $$F_N(x)= f\left(\left(N^{-1}(x_i-1/2)\right)_1^n\right),\;\; x\in B_N=\{1,\dots,N\}^n$$
 and $f$. 
 
 \begin{prop} \label{pro-flat} Fix $\epsilon,\eta>0$ and the dimension $n$.  Let $N\ge A\ge 1$. For any function $f\in\mathcal C_{n,\epsilon,\eta}(A)$, the spectral gap $\lambda$ of the associated  finite Metropolis chain on $B_N$   satisfies
$\lambda\ge  c_{n,\epsilon,\eta}/N^2 .$
  \end{prop} 
\begin{proof} For any pair of points $x,y$ in the discrete box $B_N$, pick a discrete path $\gamma_{xy}$ in $B_N$ which remains always at distance at most $\sqrt{n}/N$ of the straight line  joining the points and assume without loss of generality that $\gamma_{yx}$ is the same path in reverse. Here $\sqrt{n}/N$ is the length of the diagonal of the $n$ dimensional cube if side size $1/N$. Note that, at any point $z\in B_N$  along this discrete path, the value taken by $F_N$ at $z$, $F_N(z)$, satisfies 
$$F_N(z) \asymp_{n} f(z'_{xy})$$
where $z'_{xy}$ is the point  closest to $z'= (N^{-1}(z_i-1/2))_1^n)\in [0,1]^n$ along the straight line from $x'$ to $y'$ in $[0,1]^n$. This is because
$F_N(z)=f(z')$ and 
$$|\log (f(z')/f(z'_{xy}))|\le A\sqrt{n}/N
.$$
This will allow us to use  (a version of) the second hypothesis (\ref{C2}) along the discrete path $\gamma_{xy}$.

We now apply Proposition \ref{prop-W} with set $A=B_N$ and trivial weight $w(e)=1$. Obviously, the maximal path length is no more than $nN\asymp _n N$ and 
$$W(e)\le  \frac{nN}{Q(e)}\sum_{x\in B_N,y\in B_N \atop {\gamma_{xy}\ni e}}\pi(x)\pi(y) $$
where $\pi=Z_NF_N$ is our target probability distribution, the normalized form of $F_N$ and $Q(e)\asymp \pi(e_-)\asymp \pi(e_+)$ because of (\ref{C1}) and the fact that $A\le N$.  By hypothesis (\ref{C2}) and the remark above, we have
$$\frac{\pi(x)\pi(y)}{\pi(e_+)}\le C_n \epsilon^{-1} \sup_z \{\pi(z)\}.$$
Because of (\ref{C1})-(\ref{C3}), we  have
$$\pi(z) =Z_N f(z') \le \eta^{-1} Z_N \int_{[0,1]^n} f d\mu$$
and 
$$\int_{[0,1]^n} fd\mu \le C'_n N^{-n} \sum_{\xi\in B_N}F_N(\xi)= C'_n N^{-n} Z_N^{-1}.$$
This gives
$$\pi(z)\le  \eta^{-1} C'_n N^{-n}$$
and
$$W(e)\le  C''_n \epsilon^{-1}\eta^{-1} N^{-n}  \sum_{x\in B_N,y\in B_N \atop {\gamma_{xy}\ni e}}1.$$ 
By symmetry,
$$ \sum_{x\in B_N,y\in B_N \atop {\gamma_{xy}\ni e}}1\le 2 \max_e\left\{
  \sum_{x\in B_N,y\in B_N, d(x,e)\le d(y,e)\atop {\gamma_{xy}\ni e}}1 \right\}$$
  A simple geometric argument based on the nature of the paths $\gamma_{xy}$ (they stay close to the straight line joining $x$ to $y$), implies
  that, for any $y$ there are at most  $C'''_n N$  possible points $x$ satisfying the conditions required in the sum and thus
  $$  \sum_{x\in B_N,y\in B_N, d(x,e)\le d(y,e)\atop {\gamma_{xy}\ni e}}1 \le C'''_n N^{n+1}.$$  
  Putting all these pieces together gives
  $$W(e) \le C_n (\epsilon \eta)^{-1} N^2$$
  as desired.
  \end{proof}
  
\section{The valley effect in dimension 2}  \label{sec-valley}
We close this work  with a two-dimensional example that is reminiscent to the simplest version of the examples in Section \ref{sec-dim1}. See also \cite{Nash}.  For this example, it is convenient to work on the cube $[-1,1]^2$ and the associated discrete box 
$\mathcal B_N=\{-N+1,\dots,N\}^2$.  Because this leads to some slightly unusual coordinate shifts, we draw a small $N$ example for illustration and check. 

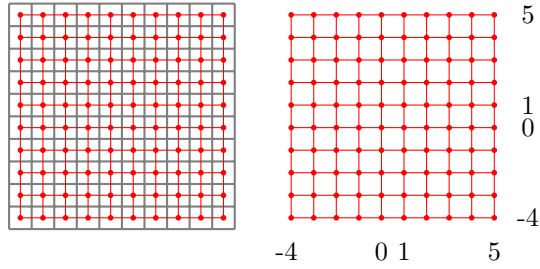
\begin{figure}[h]
\begin{tikzpicture}[scale=.3] 
\draw [help lines, thick] (-5,-5) grid (5,5);
\foreach \x in {-4.5,...,4.5}, \draw[red] (-4.5,\x) -- (4.5,\x);
 \foreach 
\x in {-4.5,...,4.5}, \draw[red] (\x,.-4.5) -- (\x,4.5);

  \foreach \x in {-4.5,...,4.5} \foreach \y in {-4.5,...,4.5} 
 \draw[fill,red] (\x,\y) circle [radius=.1];

\foreach \x in {-4.5,...,4.5}, \draw[red] (7.5,\x) -- (16.5,\x);
 \foreach 
\x in {7.5,...,16.5}, \draw[red] (\x,-4.5) -- (\x,4.5);

  \foreach \x in {7.5,...,16.5} \foreach \y in {-4.5,...,4.5} 
 \draw[fill,red] (\x,\y) circle [radius=.1];
\node at (7.3,-6) {-4}; \node at (16.5,-6) {5}; \node at (11.5,-6) {0};  \node at (12.5,-6) {1};
\node at (18,-4.5) {-4}; \node at (18,4.5) {5}; \node at (18,-.5) {0};  \node at (18,.5) {1};

\end{tikzpicture}
\caption{Grid approximation $\mathcal B_N=\{-N+1,\dots,N\}^2$ of the box $[-1,1]^2$ at level $N=5$: The red grid has $2N$ vertical lines and $2N$ horizontal lines; the grey grid decomposes $[-1,1]^2$ into $(2N)^2$ little squares.}\label{fig-V1}
\end{figure}

The model function we want to consider is 
 $$f(x,y) = (A|x+y|+1)^{\alpha}$$ where $A$ is a large constant and  $\alpha$ is a fixed non-negative constant  (in fact, the case of a negative $\alpha$ is covered by the results of the previous section).  Using the map $(x_1,x_2)\mapsto N^{-1}(x_1-1/2,x_2-1/2)$ from $\mathcal B_N$ to $[-1,1]^2$, this yields the probability measure on 
\begin{equation}\label{pialpha}
\pi = Z_NF_N\text{ with } F_N(x_1,x_2) = \left(\frac{A}{N}|x_1+x_2-1|+1\right)^{\alpha} \mbox{ on } \mathcal B_N.\end{equation}
The constant $Z_N$ is given by 
$$Z_N^{-1} = \sum_{x_1=-N+1}^N \sum_{x_2=-N+1}^N \left(\frac{A}{N} |x_1+x_2-1|+1\right)^{\alpha} \asymp A^\alpha N^2.$$ 
 Examples of  such $\pi$ are illustrated below, with $N = 5$ and $\alpha = 0.5,\;2$. 
\begin{center}
\begin{figure}   \includegraphics[width=.5\textwidth]{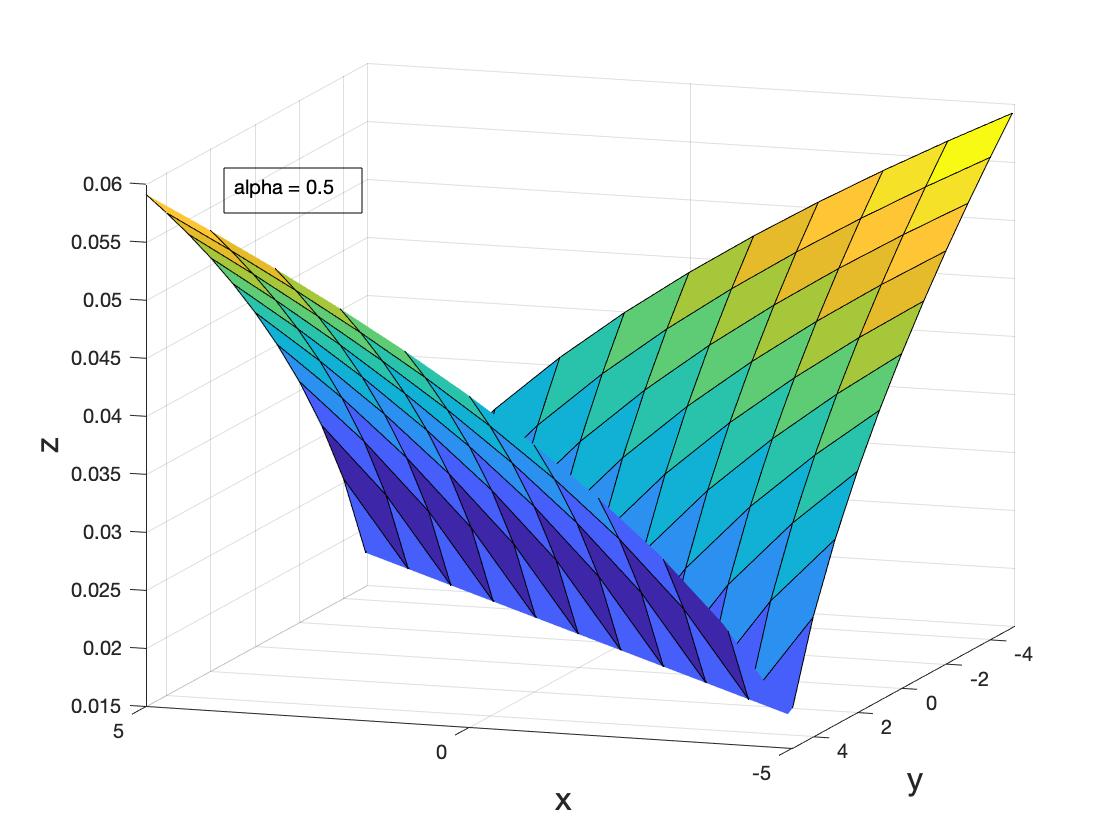}\includegraphics[width=.5\textwidth]{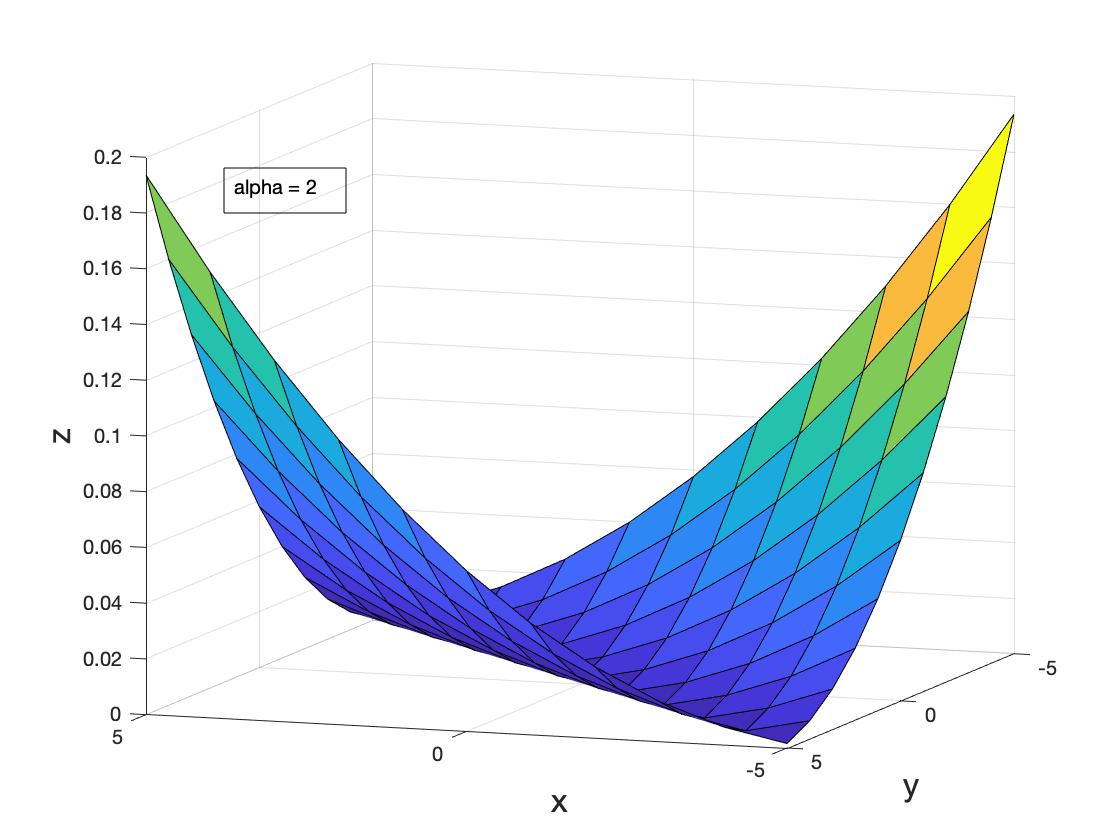}
\label{fig-V2}  \caption{The measure $\pi$ at (\ref{pialpha}) for $\alpha=0.5$ and $\alpha=2$}
\end{figure}\end{center}In fact, we will treat the following natural variation on these examples. Let $L$ be a straight line through the origin in $[-1,1]^2$ and let
\begin{equation} \label{falpha2}
f(x)= (1+ Ad_2(x,L))^\alpha, \;\alpha>0,\end{equation}
where $d_2(x,L)$ is the Euclidean distance from $x=(x_1,x_2)\in [-1,1]^2$ to the line $L$.   
As before, we set $\pi = Z_NF_N$ with \begin{equation}\label{pialpha2}
 F_N(x_1,x_2) = \left[1+Ad_2\left(\left((x_1-1/2)/N,(x_2-1/2)/N\right),L\right)\right]^{\alpha} \mbox{ on } \mathcal B_N.\end{equation}
The constant $Z_N$ is again of order $A^\alpha N^2$.

\begin{prop} Fix $\epsilon\in (0,1)$. For all  $N\in [\epsilon A, A/\epsilon]$,  the  Metropolis chain for the probability measure $\pi$ on $\mathcal B_N$ defined at  {\em (\ref{pialpha2})}has spectral gap 
$$\lambda \asymp_\epsilon \begin{cases}
\frac{1}{N^2} & \alpha < 1, \\
\frac{1}{N^2 \log(N)} & \alpha =1, \\
\frac{1}{N^{\alpha + 1}} & \alpha > 1.
\end{cases}$$
\end{prop}
\begin{proof}[Proof of the upper-bound] We explain the upper-bound in the case $A=N$ and for example (\ref{pialpha}). A similar argument works for measures of the type (\ref{falpha2})-(\ref{pialpha2}).
We have 
$$\lambda = \inf_f\left\{\frac{\mathcal{E}(f|f)}{\|f\|_2^2}\;:\; \pi(f) = 0 \text{ and } \pi(f^2)\neq 0 \right\}.$$ Consider the test function 
$$f(x)=f(x_1,x_2)= \begin{cases}
\sum_{\ell=2}^{k} \frac{1}{(\ell-1) ^\alpha} & \text{if } x_1+x_2 =k\in\{2,\dots, 2N \}\\
 0&\text{if } x_1+x_2 =1,\\-\sum_{\ell=0}^{k} \frac{1}{(1+\ell)^\alpha} & \text{if } x_1+x_2=-k \in 
\{-2N+2,\dots,0\}.
\end{cases}$$
Notice that this function $f$ is antisymmetric with respect to the diagonal  $x_1+x_2=1$ in $\{-N+1,\dots,N\}^2$. Since $\pi$ is symmetric with respect to the same diagonal, $f$ has mean $0$.  We  note that 
$$|f(x_1,x_2)|\asymp \begin{cases}
1 &\text{if } \alpha > 1 \text{ and } x_1+x_2-1\neq 0, \\
\log (1+ |x_1+x_2-1|) &\text{if } \alpha =1, \\
(1+|x_1+x_2-1|)^{1-\alpha} &\text{if } \alpha < 1 \text{ and } x_1+x_2-1\neq 0.
\end{cases}$$
It follows that
$$\|f\|_2^2
\asymp \begin{cases}
1 &\text{if } \alpha > 1, \\
(\log (1+N))^2 &\text{if } \alpha =1, \\
(1+N )^{2-2\alpha} &\text{if } \alpha < 1.
\end{cases} $$
Next, we note that for $x\sim y$ (i.e., $M(x,y)>0$) in $\mathcal B_N$, we have
$$|f(x)-f(y)|\asymp (1+|x_1+x_2-1|)^{-\alpha}.$$ This allows us to estimate  
$$\mathcal{E}(f,f) \asymp  \begin{cases}
N^{-\alpha-1} &\text{if } \alpha > 1, \\
N^{-2} \log (1+N) &\text{if } \alpha =1, \\
(1+N )^{-2\alpha} &\text{if } \alpha < 1.
\end{cases} $$
Together, these computations of $\|f\|_2^2$ and $\mathcal E(f,f)$ give the desired upper-bound on $\lambda$. 
\end{proof}

We treat the lower-bound in the case of measures of type 
(\ref{falpha2})-(\ref{pialpha2}) because this case requires a few interesting adjustments compared to  (\ref{pialpha}) but we will start with the case (\ref{pialpha}) as a warm-up. To obtain a lower-bound on $\lambda$, we return to the formula $\lambda \geq \frac{2}{W}$ as in Proposition \ref{prop-W}.

 \subsubsection*{Choice of paths} 
 Consider $\pi$ is as in (\ref{pialpha2}) with  $L=\{(x_1,x_2): ax_1+bx_2=0\}$, $a^2+b^2=1$.
Without loss of generality, we can assume that $a\ge |b|$.  This means that the line $L$  is more vertical than horizontal
because the slope is $a/b$ (vertical if $b=0$).
 
 Each pair $(x,y)\in B_N^2$ determines a rectangle $R_{xy}$ with sides parallel to the two axes.  If $x,y$ are on the same side of $L$, $\gamma_{xy}$ follows the two consecutive sides of the rectangles that are farthest away from $L$. 
If $x,y$ are on different sides of $L$ and $d_2(x,L)\le  d_2(y,L)$, we must have one of the following four configurations, $\mathfrak R_i$, $1\le i\le 4$,  which focuses on the two sides of $R_{xy}$ which intersect the line $L$.

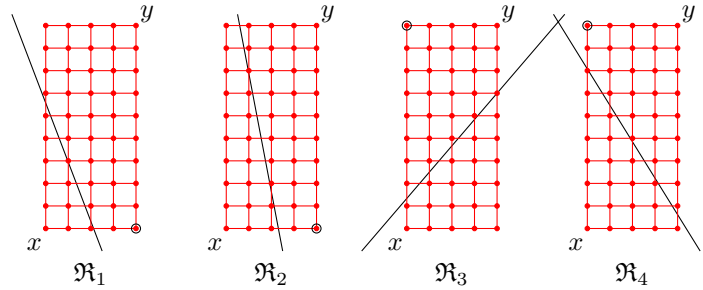
\begin{figure}[h]
\begin{tikzpicture}[scale=.3] 

\draw (10,-5.5) -- (6,5);
\foreach \x in {-4.5,...,4.5}, \draw[red] (7.5,\x) -- (11.5,\x);
 \foreach \x in {7.5,...,11.5}, \draw[red] (\x,-4.5) -- (\x,4.5);
 \foreach \x in {7.5,...,11.5} \foreach \y in {-4.5,...,4.5} 
 \draw[fill,red] (\x,\y) circle [radius=.1]; 
 \node at (9.5,-6.5) {$\mathfrak R_1$}; \node at (7,-5.2) {$x$}; \node at (12, 5) {$y$};
 \draw (11.5,-4.5) circle [radius=.2];

\draw (18,-5.5) -- (16,5);

\foreach \x in {-4.5,...,4.5}, \draw[red] (15.5,\x) -- (19.5,\x);
 \foreach \x in {15.5,...,19.5}, \draw[red] (\x,-4.5) -- (\x,4.5);
 \foreach \x in {15.5,...,19.5} \foreach \y in {-4.5,...,4.5} 
 \draw[fill,red] (\x,\y) circle [radius=.1]; 
 \node at (17.5,-6.5) {$\mathfrak R_2$}; \node at (15,-5.2) {$x$}; \node at (20, 5) {$y$};
 \draw (19.5,-4.5) circle [radius=.2]; 
 
\draw (21.5,-5.5) -- (30.5,5);

\foreach \x in {-4.5,...,4.5}, \draw[red] (23.5,\x) -- (27.5,\x);
 \foreach \x in {23.5,...,27.5}, \draw[red] (\x,-4.5) -- (\x,4.5);
 \foreach \x in {23.5,...,27.5} \foreach \y in {-4.5,...,4.5} 
 \draw[fill,red] (\x,\y) circle [radius=.1]; 
 \node at (25.5,-6.5) {$\mathfrak R_3$}; \node at (23,-5.2) {$x$}; \node at (28, 5) {$y$};
\draw (23.5,4.5) circle [radius=.2];

\draw (36.5,-5.5) -- (30,5);

\foreach \x in {-4.5,...,4.5}, \draw[red] (31.5,\x) -- (35.5,\x);
 \foreach \x in {31.5,...,35.5}, \draw[red] (\x,-4.5) -- (\x,4.5);
 \foreach \x in {31.5,...,35.5} \foreach \y in {-4.5,...,4.5} 
 \draw[fill,red] (\x,\y) circle [radius=.1]; 
 \node at (33.5,-6.5) {$\mathfrak R_4$}; \node at (31,-5.2) {$x$}; \node at (36, 5) {$y$};
\draw (31.5,4.5) circle [radius=.2];

\end{tikzpicture}
\caption{The four configrurations $\mathfrak R_i$, $1\le i\le 4$, when $x,y$ are on different sides of $L$ and $d_2(x,L)\le d_2(y,L)$. In each case the black circle show the corner used by $\gamma_{xy}$.} \label{fig-L}
\end{figure}

In all cases, $\gamma_{xy}$  starts with the side at $x$ that crosses $L$ with preference for the horizontal crossing.
If $d_2(x,L)=d_2(y,L)$ make a choice to break the tie. In all cases, $\gamma_{yx}$ is chosen to be $\gamma_{xy}$
travelled in reverse order.

\subsubsection*{Choice of weights}  Given an edge $e=(x,y)\in E$ and $\alpha>0$ as in (\ref{pialpha})-(\ref{pialpha2}), set
$$w(e)=(1+\max\{d_2(x,L),d_2(y,L)\})^{\alpha/2} \asymp F_N(x)^{1/2} \asymp F_N(y)^{1/2}.$$
Given that these choices of path and weight satisfy the assumptions in Remark \ref{symred}, in order to bound $W$, it suffices to bound $W(e)$ (for all $e$) or  (also, for all $e$)
\begin{equation}\label{Wtilde}
\widetilde{W}(e)=\frac{w(e)^2}{Q(e)} \sum_{x,y\in B_N, \pi(x)\le \pi(y) \atop {\gamma_{xy}\ni e}}|\gamma_{xy}|_w\pi(x)\pi(y).
\end{equation}
This differs from $W(e)$ by the extra condition $ \pi(x)\le \pi(y)$ in the summation.

\begin{proof}[Proof of the lower-bound on $\lambda$  for {\em (\ref{pialpha})}]  In this case, $L$ joins diagonally opposite corners.  What makes the case { (\ref{pialpha})} simpler than the general case is the computation of the $w$-length of paths. Indeed, the $w$-diameter of $B_N$ is easily computed to satisfy
$$\mbox{diam}\asymp_\alpha D(\alpha,N)= \left\{\begin{array}{cl}  1  &\mbox{ if } \alpha>1,\\
\log N &\mbox{ if } \alpha=1,\\
 N^{1-\alpha} &\mbox{ if } \alpha\in (0,1).\end{array}\right.$$
 This easily follows from computing the length of any vertical and horizontal line in $B_N$.  It is possible to track down the dependence on $\alpha$ in these computations but we do not do so. Moreover, the same holds true in the general case (\ref{pialpha2}) as long as  the angle of the line $L$ with each of the axis is bounded away from $0$ (uniformly in $N$).
 Now, because of the choice of the weight, we have
 $$W(e)\le  C N^{\alpha+2} D(\alpha,N) \sum_{x,y\in B_N  \atop {\gamma_{xy}\ni e}}\pi(x)\pi(y).$$
 It is a simple matter to very that
\begin{equation} \label{1/N}
\sum_{x,y\in B_N  \atop {\gamma_{xy}\ni e}}\pi(x)\pi(y) \le C'N^{-1}.\end{equation}
This is because the condition $\gamma_{xy}\ni e$  for some fixed $e$ places either $x$ or $y$ on a line parallel to one of the axes.  Together, these computations give
$$W(e)\le C_\alpha  \left\{\begin{array}{cl}  N^{1+\alpha}  &\mbox{ if } \alpha>1,\\
N^2 \log N &\mbox{ if } \alpha=1,\\
 N^{2} &\mbox{ if } \alpha\in (0,1).\end{array}\right.$$
and this provides the desired lower bound on $\lambda$. \end{proof}

\begin{proof}[Proof of the lower-bound on $\lambda$  when $L$ is vertical] 
To understand why the previous argument does not work in general, consider the case when $L$ is vertical.  In this case,
assuming $\pi(x)\le \pi(y)$, 
the length of a path going from $x=(x_1,x_2) $ to $y=(y_1,y_2)$ with  $|x_1|<|y_1|$ is of order
$$|\gamma_{xy}|_w\asymp_\alpha  \frac{|y_2-x_2|}{F_N(y)} + \left\{\begin{array}{cl}  1  &\mbox{ if } \alpha>1,\\
\log (1+ |y_1-x_1| ) &\mbox{ if } \alpha=1,\\
 |y_1-x_1|^{1-\alpha} &\mbox{ if } \alpha\in (0,1).\end{array}\right.$$
 Using $\widetilde{W}(e)$ instead of $W(e)$ for convenience, we write
 $$\widetilde{W}(e)\le C_\alpha (W_1(e)+W_2(e))$$
 where 
 $$W_1(e)=  N^{\alpha+2}\sum_{x,y\in B_N, \pi(x)\le \pi(y) \atop {\gamma_{xy}\ni e}} \frac{|y_2-x_2|}{F_N(y)} \pi(x)\pi(y) $$
  and (see the definition of $D(\alpha,N)$ above)
   $$W_2(e)=  N^{\alpha+2}D(\alpha,N) \sum_{x,y\in B_N, \pi(x)\le \pi(y)  \atop {\gamma_{xy}\ni e}}\pi(x)\pi(y).$$
 As above, (\ref{1/N}) gives
 $$W_2(e)\le  C'_\alpha  \left\{\begin{array}{cl}  N^{1+\alpha}  &\mbox{ if } \alpha>1,\\
N^2 \log N &\mbox{ if } \alpha=1,\\
 N^{2} &\mbox{ if } \alpha\in (0,1).\end{array}\right.$$
For $W_1(e)$, the computation is slightly different depending on whether $e$ is horizontal or vertical. In both cases, we have
$$W_1(e)\le C'_\alpha N   \sum_{x,y\in B_N, \pi(x)\le \pi(y) \atop {\gamma_{xy}\ni e}}\pi(x)\le C_\alpha'' N^2.$$
Observe that the upper-bound for $W_2(e)$ always dominates that on $W_1(e)$. The desired result follows in this case. 
This argument remains valid as long as the line $L$ crosses the top and bottom sides at bounded distance from their respective mid-point (uniformly in $N$). \end{proof}

\begin{proof}[Proof of the lower-bound on $\lambda$, in general] 
Now, consider the general case. So far, we have treated the cases $a=|b|$ (more generally, $a\asymp b$) and $b=0$ (more generally, $|b|\le C/N$ for some fixed $C$).
 
To treat the general case, we partition the pairs $(x,y)$ into three subsets:  $P_0$ is the set of pairs for which $\pi(x)\le \pi(y)$ and $x$ and $y$ lie on the same side of $L$. The second subset, $P_{12}$, is the set of pairs on different sides of $L$ for which $\pi(x)\le \pi(y)$ and $L$ is crossed horizontally. Such pairs (after the use of some symmetries) corresponds to configurations $\mathfrak R_1, \mathfrak R_2$ in Figure \ref{fig-L}. The third and last subset, $P_{34}$, is the set of pairs on different sides of $L$ for which $\pi(x)\le \pi(y)$ and $L$ is crossed vertically. Such pairs (after the use of some symmetries) corresponds to configurations $\mathfrak R_3, \mathfrak R_4$ in Figure \ref{fig-L}. We need to bound  $\widetilde{W}(e)$ at (\ref{Wtilde}) from above and we split the sum in (\ref{Wtilde}) into three parts corresponding to the the contributions of $P_0,P_{12}$ and $P_{34}$ which we call $\widetilde{W}_0(e)$, $\widetilde{W}_{12}(e)$ and $\widetilde{W}_{34}(e)$.

\subsubsection*{Contribution of $P_0$ }   The length of any paths associated with $(x,y)\in P_0$ is bounded by
$$|\gamma_{xy}|_w\le C_\alpha \left( \frac{N}{F_N(x)}+ D(\alpha,N) \right).$$
It follows that (the computation is similar to the one done above in the case when $L$ is vertical; the constant $C_\alpha$ may change from line to line) 
$$\widetilde{W}_0(e) \le C_\alpha N^2 +   C_\alpha  \left\{\begin{array}{cl}  N^{1+\alpha}  &\mbox{ if } \alpha>1,\\
N^2 \log N &\mbox{ if } \alpha=1,\\
 N^{2} &\mbox{ if } \alpha\in (0,1).\end{array}\right. $$
 Note that the second term always dominates.
 
 \subsubsection*{Contribution of $P_{12}$ }   

The contribution to $\widetilde{W}_{12}(e)$ of the horizontal $w$-length of paths is always bounded by
$$  C_\alpha  \left\{\begin{array}{cl}  N^{1+\alpha}  &\mbox{ if } \alpha>1,\\
N^2 \log N &\mbox{ if } \alpha=1,\\
 N^{2} &\mbox{ if } \alpha\in (0,1).\end{array}\right. $$
 So, we concentrate on the vertical $w$-length of the paths. For a given $e$, call the corresponding sum
 $\widetilde{W}^v_{12}(e)$ (this is the vertical component of $\widetilde{W}_{12}(e)$). 
 \begin{figure}[h]
\begin{tikzpicture}[scale=.2] 

\draw (10,-5.5) -- (6,5);
\foreach \x in {-4.5,...,4.5}, \draw[red] (7.5,\x) -- (11.5,\x);
 \foreach \x in {7.5,...,11.5}, \draw[red] (\x,-4.5) -- (\x,4.5);
 \foreach \x in {7.5,...,11.5} \foreach \y in {-4.5,...,4.5} 
 \draw[fill,red] (\x,\y) circle [radius=.1]; 
 \node at (9.5,-6.5) {$\mathfrak R_1$}; \node at (7,-5.3) {$x$}; \node at (12, 5.1) {$y$};
 \draw (11.5,-4.5) circle [radius=.2];

\draw (18,-5.5) -- (16,5);

\foreach \x in {-4.5,...,4.5}, \draw[red] (15.5,\x) -- (19.5,\x);
 \foreach \x in {15.5,...,19.5}, \draw[red] (\x,-4.5) -- (\x,4.5);
 \foreach \x in {15.5,...,19.5} \foreach \y in {-4.5,...,4.5} 
 \draw[fill,red] (\x,\y) circle [radius=.1]; 
 \node at (17.5,-6.5) {$\mathfrak R_2$}; \node at (15,-5.3) {$x$}; \node at (20, 5.1) {$y$};
 \draw (19.5,-4.5) circle [radius=.2];

\end{tikzpicture}
\caption{The configrurations $\mathfrak R_1,\mathfrak R_2$}  \label{fig-L12}
\end{figure}
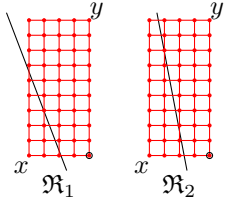
 
 Call $z$ the turning point on $\gamma_{xy}$ ($z=x$ if there is no turning point). When $(x,y)\in P_{12}$, the $w$-length of the vertical component of the path  $\gamma_{xy}$   is of order
$$ v_{xy}= \sum_0^{|y_2-x_2|} \frac{1}{(1+\sqrt{d_2(z,L)^2+ (bk)^2})^\alpha} \asymp \int _0^{|y_2-x_2|} \frac{ds}{(1+\sqrt{d(z,L)^2+ (bk)^2})^\alpha}.$$
 Hence, $v_{xy}$ is  bounded above by
 $$v_{xy}\le C_\alpha  \left\{\begin{array}{cl}  \frac{1}{|b|} \frac{1}{F_N(z)}  &\mbox{ if } \alpha>1,\\
 \frac{1}{|b|} \log\left(\frac{F_N(y)}{F_N(z)}\right)&\mbox{ if } \alpha=1,\\
  \frac{1}{|b|} \frac{d_2(y,L)}{F_N(y)}&\mbox{ if } \alpha\in (0,1).\end{array}\right.$$ 
 In the case $\alpha\in (0,1)$, $d(y,L)\le |b|N$ and it follows that 
 $v_{xy} \le N/F_N(y)$.  As discussed earlier, this gives $\widetilde{W}^v_{12}(e)\le C N^2$.
 In the case $\alpha=1$, 
 $$v_{xy}\le \frac{1}{|b|}\log \left(\frac{(1+d_2(y,L)}{1+d_2(z,L)}\right) \le C \log \left(1 + C\frac{ |b||y_2-z_2|}{1+d_2(y,L)}\right)\le C' \frac{N}{F_N(y)}.$$
 This means that the earlier computations applied again gives $\widetilde{W}^v_{12}(e)\ \le CN^2$.
 In the last case, $\alpha>1$, we consider two sub-cases. For those $x,y$ such that $d_2(z,L)\ge d_2(x,L)$, we have  $v_{xy}\le CN/F_N(x)$  and 
 $$N^{\alpha+2}\sum_{x,y\in P_{12}, d_2(x,L)\le d_2(z,L) \atop {\gamma_{xy}\ni e}} v_{xy}\pi(x)\pi(y)\le
 CN\sum_{x,y\in P_{12} \atop {\gamma_{xy}\ni e}}\pi(y) \le C'N^2.$$ For those $x,y$ such that $d_2(z,L)<d_2(x,L)$, we have $|y_1-x_1|\le bN$ and it follows that
 \begin{eqnarray*}N^{\alpha+2}\sum_{x,y\in P_{12}, d_2(z,L)\le d_2(x,L) \atop {\gamma_{xy}\ni e}} v_{xy}\pi(x)\pi(y)
& \le &C |b|^{-1}N^{\alpha+2} \sum_{x,y\in P_{12}, |y_1-x_1|\le bN, \atop {\gamma_{xy}\ni e}}\pi(x)\pi(y) \\
&\le & C' |b|^{-1} N^{\alpha+2} |b|N^{-1} \le C'N^{\alpha+1}.\end{eqnarray*}
 All together, this shows that
 $$
 \widetilde{W}^v_{12}(e)\le  C_\alpha  \left\{\begin{array}{cl}  N^{1+\alpha}  &\mbox{ if } \alpha>1,\\
 N^{2} &\mbox{ if } \alpha\in (0,1],\end{array}\right. $$ 
 and thus
 $$\widetilde{W}_{12}(e)\le   C_\alpha \left\{\begin{array}{cl}  N^{1+\alpha}  &\mbox{ if } \alpha>1,\\
N^2 \log N &\mbox{ if } \alpha=1,\\
 N^{2} &\mbox{ if } \alpha\in (0,1).\end{array}\right. $$

 \subsubsection*{Contribution of $P_{34}$}    
 As in the treatment of $\widetilde{W}_{12}$ above, the contribution of the horizontal $w$-length of the paths  to $\widetilde{W}_{34}(e)$ is always bounded by
$$  C_\alpha  \left\{\begin{array}{cl}  N^{1+\alpha}  &\mbox{ if } \alpha>1,\\
N^2 \log N &\mbox{ if } \alpha=1,\\
 N^{2} &\mbox{ if } \alpha\in (0,1).\end{array}\right. $$
 So, we concentrate on the vertical $w$-length of the paths. For a given $e$, call the corresponding sum
 $\widetilde{W}^v_{34}(e)$ (this is the vertical component of $\widetilde{W}_{34}(e)$). 
  
\begin{figure}[h]
\begin{tikzpicture}[scale=.2] 
\draw (21.5,-5.5) -- (30.5,5);

\foreach \x in {-4.5,...,4.5}, \draw[red] (23.5,\x) -- (27.5,\x);
 \foreach \x in {23.5,...,27.5}, \draw[red] (\x,-4.5) -- (\x,4.5);
 \foreach \x in {23.5,...,27.5} \foreach \y in {-4.5,...,4.5} 
 \draw[fill,red] (\x,\y) circle [radius=.1]; 
 \node at (25.5,-6.5) {$\mathfrak R_3$}; \node at (23,-5.3) {$x$}; \node at (28, 5.1) {$y$};
\draw (23.5,4.5) circle [radius=.2];
 
\draw (36.5,-5.5) -- (30,5);

\foreach \x in {-4.5,...,4.5}, \draw[red] (31.5,\x) -- (35.5,\x);
 \foreach \x in {31.5,...,35.5}, \draw[red] (\x,-4.5) -- (\x,4.5);
 \foreach \x in {31.5,...,35.5} \foreach \y in {-4.5,...,4.5} 
 \draw[fill,red] (\x,\y) circle [radius=.1]; 
 \node at (33.5,-6.5) {$\mathfrak R_4$}; \node at (31,-5.3) {$x$}; \node at (36, 5.1) {$y$};
\draw (31.5,4.5) circle [radius=.2];
\end{tikzpicture}
\caption{The configrurations $\mathfrak R_3$ and $\mathfrak R_4$} \label{fig-L34}
\end{figure}
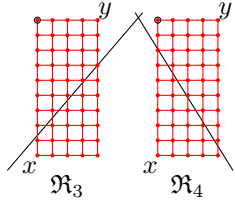
 Computing as in the case of $P_{12}$, the vertical contribution  $v_{xy}$ to the $w$-length of $\gamma_{xy}$ is  bounded above by
 $$v_{xy}\le C_\alpha  \left\{\begin{array}{cl}  \frac{1}{|b|}   &\mbox{ if } \alpha>1,\\
 \frac{1}{|b|} \log F_N(y)&\mbox{ if } \alpha=1,\\
  \frac{1}{|b|} \frac{d_2(y,L)}{F_N(y)}&\mbox{ if } \alpha\in (0,1).\end{array}\right.$$  
In addition, because  $(x,y)\in P_{34}$, we must have $d_2(y,L),d_2(x,L)\le bN$ and $|y_1-x_1|\le bN$.
  Now, in order to compute $\widetilde{W}_{34}^v(e)$ when $\alpha\in (0,1)$, we proceed exactly as before and we get
  $\widetilde{W}_{34}^v(e)\le CN^2$ as desired.   In the case $\alpha=1$, we have
\begin{eqnarray*}\widetilde{W}_{34}^v(e)&\le & CN^{\alpha+2}\sum_{x,y\in P_{34},  \atop {\gamma_{xy}\ni e}} v_{xy}\pi(x)\pi(y)\\
&\le &
 C'|b|^{-1} \log (|b|N) N^{3}\sum_{x,y\in P_{34} \atop {\gamma_{xy}\ni e}}\pi(x)\pi(y)\\ & \le & C''N^2\log(|b| N).\end{eqnarray*}
 Here, we have used the fact that the summation in $x$ is necessarily reduced to a range of order at most $|b|N$ and that, for each  $e$, either $x$ or $y$ is located along a single line parallel to an axis (as usual in these computations).
 
 Finally, in the case $\alpha>1$,
 \begin{eqnarray*}\widetilde{W}_{34}^v(e)&\le & CN^{\alpha+2}\sum_{x,y\in P_{34},  \atop {\gamma_{xy}\ni e}} v_{xy}\pi(x)\pi(y) \le
 C'|b|^{-1}N^{\alpha+2 }\sum_{x,y\in P_{34} \atop {\gamma_{xy}\ni e}}\pi(x)\pi(y)\\ & \le & C'' N^{\alpha+1}\end{eqnarray*} 
 because 
 $$\sum_{x,y\in P_{34} \atop {\gamma_{xy}\ni e}}\pi(x)\pi(y)\le C|b| N^{-1}.$$
 This last inequality uses again the fact that the summation in $x$ is necessarily reduced to a range of order at most $|b|N$ and that, for each  $e$, either $x$ or $y$ is located along a single line parallel to an axis.  
  \end{proof}
  
 \section{Concluding remarks}  Working on the cube $[0,1]^2$ and its discrete approximation, we have identified a number of simple examples of target probability densities illustrating different behavior of the Metropolis chain based on simple random walk. More precisely, we provided matching upper- and lower-bounds for the spectral gap $\lambda$ of these Metropolis chains. The spectral gap $\lambda$, and its inverse, the so-called relaxation time $1/\lambda$, are key parameters in understanding the running time of the corresponding Monte-Carlo algorithms.  However, the directly relevant parameter is the so-called mixing time, either in total variation distance, or maximum distance (or $L^2$-distance).  Namely, referring to (\ref{Ht}), the mixing time in total variation is often define as
 $$T_{\mbox{\tiny TV}}=\inf\left\{t:  \sup_x \sum_y|H_t(x,y)-\pi(y)|\le 1/e\right\}$$
  (here $e$ denotes the number $e$, the base of the natural logarithm) whereas
  $$T_\infty=\inf\left\{t:  \sup_{x,y}\left|\frac{H_t(x,y)}{\pi(y)}-1\right|\le 1/e\right\}.$$
  With this notation, it is well-known that (see, e.g., \cite{LP,StF}, $\pi_*$ is the minimal value taken by $\pi(x)$)
 \begin{equation}\label{mix}
  \frac{1}{\lambda} \le T_{\mbox{\tiny TV}}\le T_\infty  \le  \frac{(1+\log 1/\pi_*)}{\lambda}.\end{equation}
  
  To conclude this work, we review how this applies to some of our examples and what the expected true behavior should be.

 \subsection{Asymmetric one-dimensional examples of Section \ref{sec-dim1}} In these examples 
  $\log (1+1/\pi_*)  \asymp \log \max\{N_-,N_+\}$. In the symmetric case when $N_-=N_+=N$ and $a_-=a_+=a$, treated in \cite{Nash}, $ T_{\mbox{\tiny TV}}\asymp T_\infty \asymp 1/\lambda$ (the implied constants may depend on  $a$ but they do not depend on $N$; see \cite{Nash}).  One expects that the same is true in the asymmetric cases discussed here but the technique of \cite{Nash} does not apply directly: getting rid of the factor of $\log (1+1/\pi_*)$ requires some different ideas and we do not know of a directly applicable reference.

  \subsection{Exponential fall-off} In the case of the examples treated in Subsections \ref{sec-exp1}-\ref{sec-exp2}
  (assume $a\asymp b$ in Subsection (\ref{sec-exp1})), we have $\log (1+1/\pi_*)\asymp N$ and $1/\lambda \asymp 1$. Moreover, it is relatively easy to see that the lower bound $1\asymp \frac{1}{\lambda} \le T_{\mbox{\tiny TV}}$ is off by a factor of $N$. Heuristically, starting from a position that is away from the exponential peak of the distribution, it takes at least order $N$ steps to reach a set $A$ with $\pi(A)\ge 1/2$, providing the desired lower bound.  So the upper-bound $T_\infty\le C N$ is sharp and $T_{\mbox{\tiny TV}}\asymp T_\infty\asymp N$ in these cases.
  
   \subsection{Examples with spectral gap no worse than $1/N^2$ }   Examples treated in this section  can have 
   $\log (1+ 1/\pi_*)$ of order varying from $\log N$ to $N$ and it seems difficult to improve upon (\ref{mix}) without making more restrictive hypotheses.  
  
  \subsection{The valley effect} Regarding the examples treated in Section \ref{sec-valley}, the basic ideas used to treat the one-dimensional versions of these examples in \cite{Nash} do apply and it is thus possible to improve upon the lower bound in (\ref{mix}) and show that, with constants depending on the parameter $\alpha>0$ but not on $N$,
  $$\frac{1}{\lambda} \asymp T_{\mbox{\tiny TV}}\asymp T_\infty .$$
  Note that in these examples $\log(1+1/\pi_*)\asymp\log N$.
 The application of the Nash inequality technique used in \cite{Nash} to the present $2$-dimensional examples requires some work and we only indicate  briefly the main idea.  The reason such examples satisfy the proper Nash inequality, uniformly in $N$, is because each side of the valley does satisfy such and inequality and it is possible to glue together these functional inequalities to obtain one for the entire box.  When $\alpha\ge 1$, these examples present a variety of interesting behaviors including the following:  (a) from a starting point on the line $L$ (the valley), or appropriately close to it, the chain mixes in time of order $N^2$, significantly shorter than the relaxation time $1/\lambda$ (this is not  immediate to prove); (b) when starting away from the line $L$, the chain displays a quasi-stationary behavior in the sense that it appears to equilibrate after a time of order $N^2$ on about half the box (one side of the line $L$) before reaching its true equilibrium in a time of order $1/\lambda$. This requires showing that there is a large gap between the spectral gap $\lambda$ and the next eigenvalue which is of order $1/N^2\ggg \lambda$. Such effects do not occur when $\alpha\in (0,1)$. 
 See \cite[Section 6]{DHESC} for a related discussion and examples.
    
  \bibliographystyle{plain}
\bibliography{Sophie}

 \end{document}